\documentclass[12pt,a4paper,reqno]{amsart}
\usepackage{amsmath,amsfonts,amssymb,amsthm,color}  
\usepackage{mathrsfs} 
\usepackage{MnSymbol}

\usepackage[utf8]{inputenc}
\usepackage{graphicx}
\usepackage{bibentry}
\usepackage[colorlinks,citecolor=red,pagebackref,hypertexnames=false]{hyperref}

\theoremstyle{plain}
\newtheorem{theorem}{Theorem}[section]
\newtheorem{corollary}[theorem]{Corollary}
\newtheorem{lemma}[theorem]{Lemma}

\theoremstyle{definition}
\newtheorem{definition}[theorem]{Definition}

\newtheorem{remark}[theorem]{Remark}

\numberwithin{equation}{section}
\newtheorem*{theorem*}{Theorem} 

\newcommand{\R}{{\mathbb R}}
\newcommand{\Z}{{\mathbb Z}}
\newcommand{\rn}{{\mathbb R}^{n}}

\newcommand{\N}{{\mathbb N}}

\newcommand{\sint}{\strokedint}

\DeclareMathOperator{\III}{III}
\DeclareMathOperator{\II}{II}
\DeclareMathOperator{\I}{I}
\DeclareMathOperator{\dive}{div}

\DeclareMathOperator{\supp}{supp}
\DeclareMathOperator{\diam}{diam}
\DeclareMathOperator{\dist}{dist}

\DeclareMathOperator{\bmo}{BMO}

\providecommand{\trm}[1]{\textrm{#1}}

\providecommand{\ab}[1]{  \lvert  #1  \rvert }

\providecommand{\no}[1]{  \lVert  #1  \rVert }

\providecommand{\la}[1]{  \langle #1 \rangle}

\title[A duality approach]{
A duality approach to gradient H\"older estimates
for linear divergence form elliptic equations
}
\date{\today}

\author{Olli Saari} 
\author{Yuanlin Sun}
\author{Hua-Yang Wang}
\author{Yuanhong Wei}

\address{Olli Saari: Departament de Matem\`atiques,
	Universitat Polit\`ecnica de Catalunya,
	Avinguda Diagonal 647, 08028 Barcelona,
	Catalunya, Spain and Institute of Mathematics, Universitat Polit\`ecnica de Catalunya, Pau Gargallo 14, 08028 Barcelona, Catalunya, Spain}

\address{Centre de Recerca Matem\`atica, Edifici C, Campus Bellaterra, 08193 Bellaterra, Catalunya, Spain}
\email{olli.saari@upc.edu}

\address{Yuanlin Sun: School of Mathematics, Jilin University, Changchun 130012, P.R. China}
\address{Departament de Matem\`atiques,
	Universitat Polit\`ecnica de Catalunya,
	Avinguda Diagonal 647, 08028 Barcelona,
	Catalunya, Spain}
\email{yuanlin.sun98@hotmail.com}
 
\address{Hua-Yang Wang: School of Mathematical Sciences, Laboratory of Mathematics and Complex Systems, MOE, Beijing Normal University, 100875 Beijing, P.R. China}
\email{wanghuayang@amss.ac.cn}

\address{Yuanhong Wei: School of Mathematics, Jilin University, Changchun 130012, P.R. China}
\email{weiyuanhong@jlu.edu.cn}

  \makeatletter
  \@namedef{subjclassname@1991}{{\rm 2020} Mathematics Subject Classification}
  \makeatother

\usepackage[normalem]{ulem}

\begin{document}

\begin{abstract}
We show how an iteration argument inspired by sparse domination bounds can be used to deduce gradient reverse H\"older inequalities for equations with non-constant coefficients from the theory for constant coefficient equations.
We deal with coefficient matrices whose entries are either H\"older continuous or just uniformly continuous,
leading to different results.
The purpose of the approach is to highlight the connection between Schauder theory and duality of local Hardy spaces and local H\"older spaces.
In addition, we prove a sparse bound in the context of Schauder theory for divergence form elliptic partial differential equations. 
\end{abstract}

\maketitle 

\tableofcontents

\section{Introduction}
  
In this paper,
we work in an open subset $\Omega \subset \rn$ with $n \ge 2$ and study local weak solutions to divergence form elliptic partial differential equations,
functions $u \in W^{1,2}_{loc}(\Omega)$ satisfying 
\begin{equation}
\label{eq:intro-elliptic-equation}
- \dive A\nabla u = 0
\end{equation}
in $\Omega$ in the sense of distributions.
Here $A$ is a measurable matrix-valued function satisfying the standard uniform ellipticity conditions
\begin{equation}
\label{eq:intro-matrix}
\sup_{x \in \rn} |A(x)| \le \Lambda, \quad \inf_{x, \xi \in \rn} \xi \cdot A(x) \xi \ge \lambda |\xi|^{2}
\end{equation} 
for some $0 < \lambda \le \Lambda < \infty$.

The well-known result of Meyers \cite{Meyers1963} (see also \cite{Meyers1975}) 
shows that a weak solution is necessarily $W^{1,2 + \varepsilon}_{loc}(\Omega)$
regular for some $\varepsilon = \varepsilon(n,\lambda,\Lambda) > 0$.
This is most conveniently formulated as the validity of a scaling invariant reverse H\"older inequality of the gradient:
For all cubes $Q$ (which we assume to be axis parallel from this point on),
it holds  
\begin{equation}
\label{eq:intro-RHI}
\left( \sint_{Q} |\nabla u (x)|^{q} \, dx  \right)^{1/q}
\le C  \left( \sint_{2Q} |\nabla u (x)|^{2} \, dx  \right)^{1/2}
\end{equation}
for $q = 2 + \varepsilon$ and $C = C(n,\lambda,\Lambda)$;
$2Q$ is, as are all the dilations in this paper, concentric.
Values $q > 2$ even larger are admissible under additional smoothness assumptions on the matrix function $A$ in \eqref{eq:intro-elliptic-equation}.
For instance,
we can set $q = \infty$ if the coefficient matrix $A$ is smooth (see e.g.\ \cite{MR2597943}), H\"older continuous \cite{MR749677} or just Dini continuous \cite{MR2900466,MR3747493}.
An assumption that $A$ is of vanishing mean oscillation is enough for obtaining \eqref{eq:intro-RHI} for all $q < \infty$ \cite{MR1405255, MR1720770}.
The same goes for matrices $A$ with a small BMO norm \cite{MR2069724}.

In \cite{MR4794496},
it was shown that the validity of \eqref{eq:intro-RHI} implies the so-called $(2,q')$ sparse bounds for the gradient of the solution. The sparse bounds, in turn, are known to imply $L^{p}$ estimates with Muckenhoupt weights with a good estimate on the dependency on the $A_{p}$ characteristic of the weight \cite{MR3531367}.
In that sense,
they are an improvement over what is known as weighted Calder\'on--Zygmund estimates.
In the present paper,
we develop further the method of \cite{MR4794496} method,
namely,
\begin{itemize}
  \item we present a new iteration argument similar to that in \cite{MR4794496} to prove (instead of applying it) \eqref{eq:intro-RHI} using results on equations with only constant coefficients as the input, see Theorem \ref{theorem:duality-clean-bmo};
  \item we use the same iteration argument to prove 
\begin{equation}
\label{eq:intro-RHI-hoelder}
|\nabla u|_{C^{\alpha}(Q)}
\le C(n,A)  \left( \sint_{2Q} |\nabla u (x)|^{2} \, dx  \right)^{1/2}
\end{equation}  
for solutions to \eqref{eq:intro-elliptic-equation} when the matrix $A$ is $\alpha$-H\"older continuous; again using only results on equations with constant coefficients as the input, see Theorem \ref{theorem:duality-clean} and Corollary \ref{cor:holder-noncons-holder}, and the Hardy space theory.
\end{itemize} 
We also show that the sparse form argument from \cite{MR4794496} is flexible enough to include $C^{\alpha}$ theory.
Altogether,
this paper together with \cite{MR4794496} present a unified approach that allows one to deduce both Calder\'on--Zygmund estimates and Schauder estimates under essentially minimal smoothness hypotheses on the coefficients at once,
with no PDE background except for the theory for constant coefficient equations that can be found in textbooks such as \cite{MR2597943} and \cite{MR3887613}.
This approach is by no means simpler than what is commonly known, 
but the interest of our results lies in the further extension of the sparse iteration method
and the connection between Schauder theory and the theory of local Hardy spaces 
that we have not found elsewhere in the literature.


To state the sparse form estimate,
we first have to define sparse families.
\begin{definition}
\label{def:sparse}
Let $\varepsilon \in (0,1)$.
A family of cubes $\mathcal{G}$ is $\varepsilon$-sparse if for each $P \in \mathcal{G}$ there exists $E_P \subset P$ with $|E_P| \ge \varepsilon |P|$
so that 
\[
\sum_{P \in \mathcal{G}} 1_{P} \le 1.
\]
\end{definition}
In addition, we refer to Section \ref{sec:hardy-spaces} for background on the local Hardy-norms.
Then the sparse estimate relevant for the Schauder theory is the following.

\begin{theorem}
\label{intro-thm:sparse-schauder}
Let $\varepsilon \in (0,1)$.
Let $0 < \lambda \le \Lambda< \infty$. Let $\alpha \in (0,1)$
and set $p = n/(n + \alpha)$.
Let $Q$ be a cube and let $A \in C^{\alpha}(6Q;\R^{n \times n})$
satisfy \eqref{eq:intro-matrix} in $6Q$.
Let $F \in C^{\infty}(6Q;\rn)$.
Assume that $u \in W_0^{1,2}(6Q)$ satisfies for all test functions $\eta \in C^{\infty}_c(6Q)$
\[
\int A(x) \nabla u(x) \cdot \nabla \eta(x) \, dx = \int F(x) \cdot \nabla \eta(x) \, dx.
\] 

Then,
given $g \in C^{\infty}(6Q)$
and $\varphi_Q \in C^{\infty}_c(2Q)$ with 
\begin{equation} 
\label{eq:schauder-sparse}
|\partial^{\gamma} \varphi_Q| \le C_{\gamma}\ell(P)^{-|\gamma|} 
\end{equation}
for all $\gamma \in \N^{n}$,
there exists an $\varepsilon$-sparse family $\mathcal{G}$ of subcubes of $2Q$
such that  
\begin{multline*}
\left \lvert \int_{2Q} \varphi_{Q}(x) \nabla u(x) \cdot g(x) \, dx \right \rvert \\
\le   C \sum_{P \in \mathcal{G}}  |P| \left( \sint_{6P} |F(x) - \la{F}_{6P}|^{2} \, dx \right)^{1/2} \no{g}_{h_r^{p}(4P)}  .
\end{multline*} 
Here, $C = C(n,p,\lambda, \Lambda, \ell(Q)^{\alpha}|A|_{C^{\alpha}(6Q)},\varepsilon, \varphi_Q)$
is increasing in the fifth variable, and $h_r^p(4P)$ denotes the local Hardy space as defined in Definition \ref{def:local-hardy}. 
\end{theorem}

We refer to Corollary \ref{cor:schauder-from-sparse} on
how Theorem \ref{intro-thm:sparse-schauder} implies traditional $C^{1,\beta}$ bounds for all $\beta < \alpha$.
These H\"older bounds also show that the exponent $p$ in the Hardy norm cannot be lowered. A lower exponent would imply $C^{1,\alpha+\varepsilon}$ estimates for equations with $C^{\alpha}$ coefficients, something that is known to be false (see Remark \ref{remark:sharpness}).
For oscillation spaces such as $C^{\alpha}$ and BMO,
the exponent $2$ of the mean oscillations of $F$ is irrelevant. 
However,
we suspect that using more refined tools such as weak type $(1,1)$ or Hardy space space estimates as opposed to $L^{2}$ estimates, one could lower that exponent.
We do not, however,
pursue that (equally interesting) problem here.

Finally,
the reader not so familiar with local Hardy spaces 
may appreciate the simplified yet less efficient version of our argument for reverse H\"older inequalities given in Section \ref{sec:bmo-case}.
There,
we use a variant of the scheme leading to Schauder estimates as above
but replacing the Hardy--H\"older duality by a plain application of H\"older's inequality.
Such an argument is strong enough to give reverse H\"older inequality \eqref{eq:intro-RHI} at small scales for equations whose coefficients are uniformly continuous,
but it is not good enough to provide $\bmo$, $L^{\infty}$ or  H\"older bounds.

\subsection*{Comparison with the literature}
The by-now classical argument that is commonly used for proving Calder\'on--Zygmund estimates for various equations goes back to \cite{MR1486629}.
In that paper,
an argument using a good-lambda argument is given.
Also in that case, \eqref{eq:intro-RHI} plays a crucial role. 
Coupled with a clever application of Chebyshev's inequality,
it gives a gain that is needed to run the good-lambda argument.
To make a comparison,
the arguments here and in \cite{MR4794496} replace the measure theoretic inequality (Chebyshev) by a functional analytic one (H\"older's inequality or Hardy--H\"older duality). 
This approach will allow us to treat a variety of function spaces under (most probably) minimal coefficient regularity following ideas of \cite{MR3625108,MR4058547}.
On the other hand, 
for the time being,
the approach based on duality pairing seems to have very limited applicability in the realm of non-linear equations,
where measure theoretic methods are very efficient, see e.g.\ \cite{MR4359452} for Schauder estimates on inhomogeneous but constant coefficient systems and \cite{MR4665778} for Schauder estimates on variational problems with highly irregular coefficients.
One more approach to a variety of regularity estimates
is the way of potential estimates,
see \cite{MR2900466} for a somewhat complete set of results for a number of equations (including the ones here) when the right hand side is a measure.
Finally,
the reader interested in the classical approach to Schauder estimates can consult Chapter 6 in \cite{MR1814364}.

We also mention \cite{MR3897969},
where the authors derive $(1,1)$ sparse estimates for global Leray projectors using Riesz transforms associated to operators. 
They also study sparse estimates involving $L^{p}$ averages (for any $p >0$)
of certain operator-adapted square functions,
from which global operator-adapted Hardy space bounds follow (Remark 7.7 in \cite{MR3897969}).
However, the abstract assumptions in \cite{MR3897969} on the differential operators (injective, sectorial and maximal accretive on $L^{2}$; admitting holomorphic $H^{\infty}$ functional calculus on $L^{2}$; generating a semigroup subject to Davies--Gaffney $L^{2}-L^{2}$ off-diagonal estimates and having a kernel representation with pointwise Gaussian bounds with constants uniform in $t$) do not include divergence form elliptic operators with H\"older continuous coefficients (see Remark 4.24 in \cite{MR1600066} for failure of uniform Gaussian estimates).
We also remark that the coincidence of the operator adapted Hardy spaces with the usual ones is a subtle question, especially for $p \in (0,1)$ 
(see Theorem 6.1 in \cite{MR2395177} for some concrete positive cases and section 9 in \cite{MR4628043} for a more general and complete discussion).  

When it comes to local Hardy spaces in general \cite{MR1253178,MR1223705,MR447953,MR523600},
there is a vast literature (not admitting a complete review here),
including conditions on boundedness of Calder\'on--Zygmund operators \cite{MR4475438,MR4758469},
div-curl lemmas of various kinds \cite{MR2515406,MR3342492}; 
applications to partial differential equations and much more.
As the works closest to our topic,
we mention 
the estimates for the second derivative for solutions to non-divergence form equations \cite{MR2216900};
the equivalence of H\"older type estimates on kernels and 
Hardy bounds for resolvents in the theory of block form elliptic systems \cite{MR4628043},
and the results on maximal regularity in \cite{MR2384540}.

\bigskip

\noindent \textbf{Acknowledgment.}
Olli Saari is supported by the Spanish State Research Agency MCIN/AEI/10.13039/501100011033, Next Generation EU and by ERDF “A way of making Europe” through the grants RYC2021-032950-I, PID2021-123903NB-I00 and the Severo Ochoa and Maria de Maeztu Program for Centers and Units of Excellence in R\&D, grant number CEX2020-001084-M. Yuanhong Wei is supported by the National Natural Science Foundation of China (Grant No. 12571120, 12271508),  and Scientific Research Project of Education Department of Jilin Province.
We thank Fr\'ed\'eric Bernicot and Moritz Egert for helpful explanations about \cite{MR3897969}.

\section{Notational conventions}
If not otherwise stated, 
constants $C$ are allowed to depend on the parameters
as specified in the statement of the theorem in the proof of which they appear.
For inequalities involving such constants,
like 
\[
a \le Cb, \quad b \le Ca, \quad \frac{1}{C} a \le b \le C a
\]
we occasionally use notations
\[
a \lesssim b , \quad b \lesssim a , \quad a \sim b.
\]
Given a matrix $A \in \mathbb{R}^{n \times n}$, we denote its transpose by $A^T$ and the set of its singular values by $\sigma(A)$.
For $E \subset \rn$ a measurable set,
we denote its Lebesgue measure by $\ab{E}$.
Also, for $f \in L^{1}_{loc}(\rn)$,
we denote 
\[
 \frac{1}{|E|} \int_{E} f(x) \, dx = f_{E} = \la{f}_{E} = \sint_{E} |f(x)| \, dx, 
\]
in function which notation suits best the local typography.
We denote for $x \in \rn$
\[
\ab{x}_{\infty} = \max_{1 \le i \le n} |x_i|, \quad \dist_{\infty}(x,E) = \inf\{\ab{x-y}: y \in E\}.
\] 

For cubes,
we write $Q(x,r) = \{ y \in \rn: \ab{x-y}_{\infty} < r \}$
and we write $NQ(x,r):= Q(x,Nr)$ whenever $N > 0$.
Moreover,
we set $c(Q(x,r))=x$ and $\ell(Q(x,r)) = 2r$.
If $Q$ is a cube,
its dyadic subcubes or cubes dyadic with respect to it are $P\subset Q$ such that there exists $N > 0$ and $x$ such that if
\[
Q \in \{ r 2^{k}((0,1)^{n} + j ) + x: j\in \Z^n, k \in \Z  \}
\]
then 
\[
P \in \{ r 2^{k}((0,1)^{n} + j ) + x: j\in \Z^n, k \in \Z  \}.
\]
The Hardy--Littlewood maximal function of a locally integrable function is 
\[
Mf(x) := \sup_{z \in \rn: \ r > 0} \ 1_{Q(z,r)}(x) \sint_{Q(z,r)}|f(y)|\, dy.
\] 
We will use that this operator is bounded $L^{p}(\rn) \to L^{p}(\rn)$ for $p\in (1,\infty)$
with norm bounded by $C(n,p)$
and that there exists $C(n) > 0$ such that for all $\lambda > 0$
\[
|\{x \in \rn: Mf(x) > \lambda \}| \le \frac{ C(n)}{\lambda} \int_{\rn} |f(x)|\,dx
\]
for all $f \in L^{1}_{loc}(\rn)$.

\section{Local Hardy spaces}
\label{sec:hardy-spaces}
Let $\phi: \R^{n} \to \R^{n}$ be the standard mollifier defined through 
\[
\phi(x) =  c_{n} 1_{Q(0,1)}(x) \prod_{i=1}^{n} e^{-\frac{1}{1-|x_i|^{2}}}
\] 
where $c_{n}$ is a constant guaranteeing $\no{\phi}_{L^{1}(\rn)} = 1$.
We denote $\phi_{s}(x) = s^{-n} \phi(x/s)$ so that $\phi_s$ is a smooth function with $\supp \phi_s = \overline{Q(0,s)}$ and the family $\{\phi_s: s > 0\}$ is an approximation to the identity as $s \to 0$.
The local smooth maximal operator $\mathcal{M}_s$ is defined by setting 
\[
\mathcal{M}_s f(x) = \sup_{0 < r < s/2} | \phi_r * f(x)|
\]
whenever $f$ is a distribution.
We define the local Hardy spaces following Goldberg \cite{MR523600}
and Chang--Krantz--Stein \cite{MR1223705}. 

\begin{definition}
\label{def:local-hardy-unit-scale}
Let $p \in (0,1]$.
For a distribution $f$,
we define 
\[
\no{f}_{h^{p}(\rn)} = \no{ \mathcal{M}_2 f}_{L^{p}(\rn)}.
\]
We define the local Hardy space as the class of distributions for which this norm is finite, that is,
\[
h^{p}(\rn) = \{f: \no{f}_{h^{p}(\rn)} < \infty \}. 
\] 
We define 
\[
h_r^{p}(Q(0,1)) = \{f \rvert_{Q(0,1)}:  f \in h^{p}(\rn) \}  
\]
and  
\begin{equation}
\label{eq:def-local-hardy}
\no{f}_{h_r^{p}(Q(0,1))} := \inf_{ \tilde{f} \in L^{p}(\rn), \tilde{f}\rvert_{Q(0,1)} = f}  \no{ \mathcal{M}_2 \tilde{f}}_{L^{p}(\rn)}.
\end{equation}  
In addition,
we define the local Hardy space with vanishing trace as  
\[
h_z^{p}(Q(0,1)) = \{f \in h^{p}(\rn) : \ f = 1_{Q(0,1)}f \}  
\]
and 
\begin{equation}
\label{eq:def-local-hardy-van}
\no{f}_{h_z^{p}(Q(0,1))} := \no{1_{Q(0,1)} f}_{h^{p}(\rn)} .
\end{equation} 
\end{definition}


\begin{definition}
\label{def:holder-space}
Let $\alpha \in [0,1)$.
Let $f \in L_{loc}^{2}(Q(0,1))$.
We define 
\begin{multline*}
\no{f}_{\Lambda_{z}^{\alpha}(Q(0,1))} \\
= \sup_{\substack{z \in Q(0,1) \\ 0< 4r < \dist_{\infty}(z,\partial Q(0,1)) }} \left( \inf_{c \in \R} \frac{1}{r^{2\alpha}} \sint_{Q(z,r)} |f(y) - c|^{2} \, dy \right)^{1/2}  \\
+ \sup_{\substack{z \in Q(0,1) \\  2r < \dist_{\infty}(z,\partial Q(0,1))< 4r }} \left(  \frac{1}{r^{2\alpha}} \sint_{Q(z,r)} |f(y)|^{2} \, dy \right)^{1/2}
\end{multline*}
and 
\[
\Lambda_{z}^{\alpha}(Q(0,1)) = \{f \in L^{2}_{loc}(Q(0,1)) :  \no{f}_{\Lambda_{z}^{\alpha}(Q(0,1))} < \infty \}.
\]
We also define 
\[
\no{f}_{\Lambda_{r}^{\alpha}(Q(0,1))} 
= \sup_{\substack{z \in Q(0,1) \\ 0< r < \dist_{\infty}(z,\partial Q(0,1)) }} \left( \inf_{c \in \R} \frac{1}{r^{2\alpha}} \sint_{Q(z,r)} |f(y) - c|^{2} \, dy \right)^{1/2}
\]
and
\[
\Lambda_{r}^{\alpha}(Q(0,1)) = \{f \in L^{2}_{loc}(Q(0,1)) :  \no{f}_{\Lambda_{r}^{\alpha}(Q(0,1))} < \infty \}.
\]
\end{definition}

  The spaces $\Lambda_{z}^{\alpha}(Q(0,1))$ and $\Lambda_{r}^{\alpha}(Q(0,1))$ are spaces of H\"older continuous functions.
We write the definition based on $L^{2}_{loc}(\rn)$ hypothesis at the background,
but due to Campanato's theorem, the definition immediately implies H\"older continuity of order $\alpha$.
Moreover,
the second term in the $\Lambda_z^{\alpha}$ norm forces the functions to vanish at the boundary of {$Q(0,1)$},
and it is clear by inspection that 
\[
1_{\{\trm{$f = 0$ at $\partial Q(0,1)$}\}}(f) \no{f}_{\Lambda_{z}^{\alpha}(Q(0,1))} +  \no{f}_{\Lambda_r^{\alpha}(Q(0,1))}  \le c_{n,\alpha} |f|_{C^{\alpha}(Q(0,1))}.
\] 

The following theorem is due to Chang \cite{MR1253178} (Theorem 2.1).
It builds on the atomic decomposition from \cite{MR1223705}.

\begin{theorem} 
\label{thm:duality-unit-scale}
{Let $p  \in (n/(n+1),1]$ and $\alpha = n(1/p-1)$.}
Let $a \in \{z,r\}$ and $b \in \{z,r\} \setminus \{a\}$.
\begin{itemize}
  \item If $L : h_{a}^{p}(Q(0,1)) \to \R$ is a bounded linear functional,
  then there exists $g \in \Lambda_{b}^{\alpha}(Q(0,1))$ such that for all $f \in L^{2}(Q(0,1))$ it holds 
\[
L f = \int g(y) f(y) \, dy 
\]
and $\no{g}_{\Lambda_{b}^{\alpha}(Q(0,1))} \sim \no{L}_{h_{a}^{p}(Q(0,1)) \to \R} $.
Here the implicity constants only depend on $p$ and $n$.
\item If $g \in \Lambda_{b}^{\alpha}(Q(0,1))$, then for all $f \in L^{2}(Q(0,1))$
it holds 
\[
\left \lvert \int g(y) f(y) \, dy \right \rvert \le c_{n,p}  \no{g}_{\Lambda_{b}^{\alpha}(Q(0,1))}\no{f}_{h_a^{p}(Q(0,1))}.
\] 
\end{itemize}
\end{theorem}

Next we state the scaled versions of the definitions of local Hardy spaces and the duality theorem.
The point here is to make the constants appearing in the theorems independent of the domain.
For this purpose,
we define Hardy spaces on polytopes as affine pull-backs of Hardy spaces in the unit cube.
For the special case of affine mappings preserving cubes,
we reserve the following notation.
Given a point $x \in \rn$ and a scale $s > 0$,
we define 
\[
\delta_{x,s}(y) = s(y-x)
\]
so that 
\[
\delta_{x,s}(Q(0,1)) = Q(x,s). 
\] 

\begin{definition}\label{def:local-hardy}
Let $p \in (0,1]$.
Let $A$ be an invertible affine mapping $\rn \to \rn$.
Let $a \in \{z,r\}$.
Then we say $f \in h_a^{p}(AQ(0,1))$ if $f \circ A^{-1} \in h_a^{p}(Q(0,1))$
and we define 
\[
\no{f}_{h_a^{p}(AQ(0,1))} = \no{f \circ A^{-1}}_{h_a^{p}(Q(0,1))}.
\]
Similarly,
we say $f \in \Lambda_a^{p}(AQ(0,1))$ if $f \circ A^{-1} \in \Lambda_a^{p}(Q(0,1))$
and we define 
\[
\no{f}_{\Lambda_a^{p}(AQ(0,1))} = \no{f \circ A^{-1}}_{\Lambda_a^{p}(Q(0,1))}.
\]
\end{definition}
By change of variable and the definition \eqref{eq:def-local-hardy}
we see that for $a \in \{z,r\}$
\begin{align} 
\label{eq:def-local-hardy-true-rest}
\no{f}_{h_r^{p}(Q(x,s))} &= \inf_{ \tilde{f} \in L^{p}(\rn), \tilde{f}\rvert_{Q(x,s)} = f} \left( \frac{1}{s^{n}} \int  \mathcal{M}_{s}\tilde{f}(y) ^{p} \, dy \right)^{1/p}, \\
\label{eq:def-local-hardy-true-zero}
\no{f}_{h_z^{p}(Q(x,s))} &=   \left( \frac{1}{s^{n}} \int  \mathcal{M}_{s}(1_{Q(x,s)} f)(y) ^{p} \, dy \right)^{1/p}.
\end{align} 
Similarly,
by a change of variable,
we can state the following corollary of Chang's theorem.
\begin{corollary}
\label{cor:duality-local-hardy}
Let $A$ be an invertible affine mapping.
Let $0<\lambda< \Lambda$ be the smallest and largest of the singular values of its linear part.
Let $Q = AQ(0,1)$.
Let $p  \in (n/(n+1),1]$ and $\alpha = n(1/p-1)$.
Let $a \in \{z,r\}$ and $b \in \{z,r\} \setminus \{a\}$.
\begin{itemize}
  \item If $L : h_{a}^{p}(Q) \to \R$ is a bounded linear functional,
  then there exists $g \in \Lambda_{b}^{\alpha}(Q)$ such that for all $f \in L^{2}(Q)$ 
it holds 
\[
L f = \sint_{Q} g(y) f(y) \, dy 
\]
and $\diam(Q)^{\alpha} \no{g}_{\Lambda_{b}^{\alpha}(Q)} \sim \no{L}_{h_{a}^{p}(Q) \to \R} $.
Here the implicit constants only depend on $p$, $n$, $\lambda$ and $\Lambda$.
\item If $g \in \Lambda_{b}^{\alpha}(Q)$, then for all $f \in L^{2}(Q)$ 
it holds 
\[
\left \lvert \sint_{Q} g(y) f(y) \, dy \right \rvert \le c_{n,p,\lambda,\Lambda}  \diam(Q)^{\alpha} \no{g}_{\Lambda_{b}^{\alpha}(Q)}\no{f}_{h_{a}^{p}(Q)}.
\] 
\end{itemize}
\end{corollary}

We will also need the grand maximal function characterization of the local Hardy spaces. 
For that, and also other purposes,
we set a notation for normalized bump functions.
\begin{definition}
\label{def:normalized-bumps}
Let $Q$ be a cube.
We define 
\[
\mathcal{A}_{Q} = \{ \varphi \in C^{\infty}_c (Q): \sum_{|\gamma| \le N_0} |\partial^{\gamma} \varphi| \le \ell(Q)^{-|\gamma|}  \}
\]
for a value $N_0 = N_0(n,p)$ which will remain fixed for the paper.
This is the number of derivatives required for the grand maximal function characterization of both global and local Hardy spaces, as stated in Theorem 11 of \cite{MR447953} (and thus in Theorem 1 of \cite{MR523600}).
\end{definition}

We define the local grand maximal function
\[
\mathcal{M}_{s,*} f(x) := \sup_{t \in (0,s/2)} \sup_{\varphi \in \mathcal{A}_{Q(0,1)}} |\varphi_{t} * f(x)|.
\]
By a change of variables,
Theorem 1 in \cite{MR523600} yields the following result.
\begin{lemma}
\label{lemma:grand-function} 
Let $Q$ be a cube, $p  \in (n/(n+1),1]$, and let $f \in h^{p}(\rn)$.
Then
\[
\int \mathcal{M}_{\ell(Q)} f(x)^{p} \, dx \sim \int \mathcal{M}_{\ell(Q),*} f(x)^{p} \, dx
\]
with the implicit constant only depending on $n$ and $p$.
\end{lemma}

Finally,
we will need a variant of a special case of Theorem 4 in \cite{MR523600}
assuring that multiplication by a cut-off induces a bounded operator on local Hardy spaces.

\begin{lemma}
\label{lemma:cut-off-hardy}  
Let $Q$ be a cube and $p  \in (n/(n+1),1]$.
Let $\psi \in \mathcal{A}_{Q}$.

If $f \in h_r^{p}(Q)$
and $\tilde{f} \in L^{p}(\rn)$ satisfies $\tilde{f} = f$ in $Q$,
then
\begin{align*}
\no{\psi f}_{h_z^{p}(Q)} &\le C(n,p)  \no{\mathcal{M}_{\ell(Q),*}(1_{Q}f)}_{L^{p}(\rn)}, \\
\no{\psi f}_{h_r^{p}(Q)} &\le C(n,p)  \no{\mathcal{M}_{\ell(Q),*} \tilde{f}}_{L^{p}(\rn)} . 
\end{align*}
In particular, multiplication by $\psi$ is a bounded operator both in $h_{r}^{p}(Q)$ and in $h_{z}^{p}(Q)$ with norm bounded by a constant only depending on $p$ and $n$.
\end{lemma}
\begin{proof}
As $\psi \in \mathcal{A}_Q$,
we have that $\tilde{\psi}_{x_0,s}$ defined through 
\[
\tilde{\psi}_{x_0,s}(x) = \psi(x_0-s x)
\]
has the derivative bounds as a function in $\mathcal{A}_{Q(0,1)}$ for all $x_0 \in \rn$ and $s \in (0, \ell(Q)/2)$.
Hence for all $\varphi \in \mathcal{A}_{Q(0,1)}$ and $s < \ell(Q)/2$,
we have 
\begin{multline}
\label{eq:proof-cut-off-lemma}
|\varphi_{s} * (\psi \tilde{f}) (x)|
= \left \lvert \int \frac{1}{s^{n}} \varphi\left(\frac{x-y}{s}\right) \psi(y) \tilde{f}(y) \, dy \right \rvert \\
= \left \lvert \int \frac{1}{s^{n}} \varphi\left(\frac{x-y}{s}\right) \tilde{\psi}_{x,s} \left(\frac{x-y}{s} \right) \tilde{f}(y) \, dy \right \rvert \\
\le C \sup_{\tilde{\varphi} \in \mathcal{A}_{Q(0,1)}} | \tilde{\varphi}_{s}*\tilde{f}(x)|. 
\end{multline}
By definition
\[
\no{\psi f}_{h_r^{p}(Q)}^{p}
\le \ell(Q)^{-n} \int_{\rn}  \mathcal{M}_{\ell(Q)}( \psi \tilde{f})(x)^{p} \, dx
\]
and consequently 
\[
\no{\psi f}_{h_r^{p}(Q)}^{p} \le C \ell(Q)^{-n} \int_{\rn} \mathcal{M}_{\ell(Q),*}\tilde{f}(x)^{p} \,dx .
\]
The operator norm bound follows by applying Lemma \ref{lemma:grand-function} and taking infimum over all extensions $\tilde{f}$.
Similarly, setting $\tilde{f} =  1_{Q(x,s)}f$,
we see that 
\[
\no{\psi f}_{h_z^{p}(Q)}^{p} \le C \ell(Q)^{-n} \int_{\rn} \mathcal{M}_{\ell(Q),*}(1_{Q}f(x))^{p} \,dx.
\]
\end{proof}

\section{Preliminaries on constant coefficient equations}
   
We start by a Schauder estimate for constant coefficient equations.
For a matrix $A \in \R^{n\times n}$ with smallest singular value $\lambda > 0$ and largest singular value $\Lambda < \infty$,
we consider its symmetric part $\tilde{A} = (A+A^{T})/2$ whose eigenvalues also lie in $[\lambda,\Lambda]$. 
Consider the square root $\tilde{A}^{1/2}$ of $\tilde{A}$,
and define  the dilation factors
\[
k_{\lambda,\Lambda} :=3 \sqrt{ \frac{ n}{\lambda}}  \ge \inf \{ \delta > 0: Q(0,3) \subset  \tilde{A}^{1/2}Q(0,\delta)\}
\]
and 
\[
K_{\lambda,\Lambda} := 3\sqrt{\frac{ n \Lambda }{ \lambda}} \ge \sup \{ |x|_{\infty} :  x \in \tilde{A}^{1/2}Q(0,k_{\lambda,\Lambda}) \}.
\]
For a generic axis parallel cube $Q$,
we let $S_{A}(Q)$ be a translate of $\tilde{A}^{1/2}k_{\lambda,\Lambda} Q $ that contains $3Q$.
For future reference, we also notice that we have 
\begin{equation}
\label{cube-containing-polytope}
K_{\lambda,\Lambda} Q \supset S_A(Q).
\end{equation}

For any smooth vector field $F$ in $\tilde{A}^{1/2}Q$, 
the solutions $u$ to  
\[
-\dive A \nabla u = \dive F
\] 
in $\tilde{A}^{1/2}Q$ can be seen to satisfy 
\[
-\Delta (u \circ \tilde{A}^{1/2}) = \dive \tilde{F}
\]
in $Q$, where 
$\tilde{F}(x) =  (\tilde{A}^{1/2})^{-1} F(\tilde{A}^{1/2}x)$.
In particular, 
if $u \in W^{1,2}_0(AQ)$, $x \in \partial \tilde{A}^{1/2}Q$ and $F = 0$ in $B(x,2r) \cap \overline{\tilde{A}^{1/2}Q}$ for some $r > 0$,
then $u \circ A$ has a harmonic extension $\tilde{u}$ to a neighborhood of $(\tilde{A}^{1/2})^{-1}x$ by odd reflections. Consequently, boundary regularity estimates for $u$ in $B(x,r)$ follow from interior regularity estimates for $\tilde{u}$.
Now,
passing to the Schauder estimate below follows along the lines of Theorem 10.1 in CVGMT version of the lecture notes \cite{MR3887613}. 

\begin{lemma}
\label{lemma:hoelder-constant} 
Let $A_0$ be a (constant) matrix with smallest singular value $\lambda > 0$ and largest singular value $\Lambda$.
Let $Q$ be a cube and let $u \in W_0^{1,2}(S_{A_0}(Q))$ be a weak solution to 
\[
- \dive A_0 \nabla u(x) = \dive F(x)
\]
in $S_{A_0}(Q)$ for $F \in C^{2}(S_{A_0}(Q);\rn)$.

Let $\alpha \in (0,1)$.
Then for all $x,y \in S_{A_0}(Q)$
\[
|\nabla u(x) - \nabla u (y)| 
\le C  |x-y|  ^{\alpha}    \ab{F}_{C^{\alpha}(S_{A_0}(Q);\rn)}
\]
where $C = C(n,\alpha,\lambda,\Lambda)$.
\end{lemma} 

In addition to the global Schauder estimate,
we will need a local H\"older estimate for the gradient.
This is a straigthforward consequence of, say, Lipschitz estimate for solutions,
as derivatives of solutions to constant coefficient equations are solutions also themselves.
The lemma below hence follows, for instance, from Theorem 2 in Section 6.3 of \cite{MR2597943}.

\begin{lemma}
\label{lemma:hoelder-constant-no-RHS} 
Let $A_0$ be a (constant) matrix with smallest singular value $\lambda > 0$ and largest singular value $\Lambda$. 
Let $Q$ be a cube.
Let $u \in W^{1,2}(3Q)$ be a (local) weak solution to 
\[
- \dive A_0 \nabla u(x) = 0.
\] 

Let $\alpha \in [0,1)$.
Then for all $x,y \in Q$
\[
|\nabla u(x) - \nabla u (y)| 
\le C   \left( \frac{|x-y|}{\ell(Q)}\right)^{\alpha} \left( \sint_{2Q} |\nabla u(x)|^{2} \, dx \right)^{1/2}
\]
where $C = C(n,\alpha,\lambda,\Lambda)$.
\end{lemma} 

Dual to the global H\"older estimate from Lemma \ref{lemma:hoelder-constant},
we have a local Hardy space estimate. 

\begin{lemma}
\label{lemma:hardy-constant} 
Let $A_0$ be a (constant) matrix with smallest singular value $\lambda > 0$ and largest singular value $\Lambda$.
Let $p \in (n/(n+1),1)$ and $\alpha = n(1/p-1)$.
Let $Q$ be a cube, $\Omega = S_{A_0}(Q)$ and let $u \in W_0^{1,2}(\Omega)$ be a weak solution to 
\[
- \dive A_0 \nabla u  = \dive B F 
\]
for $F \in L^{2}(\Omega;\rn)$ and $B \in C^{\alpha}(\Omega;\R^{n \times n})$ with $\alpha \in (0,1)$. 

Then it holds 
\[
\no{\nabla u}_{h_{z}^{p}(\Omega;\rn)} \le C (\no{B^{T}}_{L^{\infty}(\Omega;\mathbb{R}^{n \times n})} + \diam(\Omega)^{\alpha} \ab{B^{T}}_{C^{\alpha}(\Omega;\mathbb{R}^{n \times n})}) \no{F}_{h_{z}^{p}(\Omega;\rn)}
\] 
where $C = C(n,p,A_0)$.
If in addition, $\supp B \subset \Omega$,
then 
\[
\no{\nabla u}_{h_{r}^{p}(\Omega;\rn)} \le C (\no{B^{T}}_{L^{\infty}(\Omega;\mathbb{R}^{n \times n})} + \diam(\Omega)^{\alpha} \ab{B^{T}}_{C^{\alpha}(\Omega;\mathbb{R}^{n \times n})}) \no{F}_{h_{r}^{p}(\Omega;\rn)}.
\] 
\end{lemma}

\begin{proof}
We argue by duality.
Let $a \in \{z,r\}$ and $b \in \{z,r\} \setminus \{a\}$.
Let $\alpha = n(1/p-1)$ and let $g \in \Lambda_{b}^{\alpha}(\Omega;\rn)$.
Then $g \in L^{2}(\Omega;\rn)$ so that $\dive g \in W^{-1,2}(\Omega)$.
Let $w \in W^{1,2}_{0}(\Omega)$ be the solution (which exists, by Lax--Milgram theorem) to 
\[
\dive A^{T}_0\nabla w = \dive g
\]
so that $A^{T}_0\nabla w - g$ is divergence free.
Because $u$ vanishes on $\partial \Omega$ in the Sobolev sense,
we can use this and later the equation for $u$ to obtain
\begin{multline*}
\left \lvert \int_{\Omega} \nabla u(x) \cdot g(x) \, dx \right \rvert
= \left \lvert \int_{\Omega} A_0\nabla u(x) \cdot \nabla w(x) \, dx \right \rvert \\
= \left \lvert \int_{\Omega} B(x) F(x) \cdot \nabla w(x) \, dx \right \rvert.
\end{multline*}
By Corollary \ref{cor:duality-local-hardy} and Lemma \ref{lemma:hoelder-constant}
the right hand side is bounded by 
\[
C \diam(\Omega)^{\alpha+n} \no{F}_{h_a^{p}(\Omega;\rn)} \no{B^{T} \nabla w}_{\Lambda_{b}^{\alpha}(\Omega;\rn)}. 
\] 
Using the Campanato characterization of H\"older norms
(and the boundary values of $B$ for $a = r$),
we see 
\begin{multline*}
\no{B^{T} \nabla w}_{\Lambda_{b}^{\alpha}(\Omega;\rn)}
\le \ab{B^{T} \nabla w}_{C^{\alpha}(\Omega;\rn)} \\
\le \no{B^{T}}_{L^{\infty}(\Omega;\R^{n \times n})}\ab{\nabla w}_{C^{\alpha}(\Omega;\rn)} + \no{\nabla w}_{L^{\infty}(\Omega;\rn)}\ab{B^{T}}_{C^{\alpha}(\Omega;\R^{n \times n})} .
\end{multline*} 
Because $A_0$ is constant,
$w$ solves also the equation $\dive A_0^{T}\nabla w = \dive (g-\la{g}_{\Omega})$.
Hence 
\begin{multline*}
\no{\nabla w}_{L^{\infty}(\Omega;\rn)}
\le |\la{{|}\nabla w{|}}_{\Omega}| + \diam(\Omega)^{\alpha}\ab{\nabla w}_{C^{\alpha}(\Omega;\rn)} \\
\le \left( \sint_{\Omega}|g(x)-\la{g}_{\Omega}|^{2} \, dx \right)^{1/2}  + \diam(\Omega)^{\alpha}\ab{\nabla w}_{C^{\alpha}(\Omega;\rn)}
\end{multline*}
so that by Lemma \ref{lemma:hoelder-constant}
\begin{multline*}
\no{B^{T} \nabla w}_{\Lambda_{b}^{\alpha}(\Omega;\rn)} \le C(\no{B^{T}}_{L^{\infty}(\Omega;\R^{n \times n})} \\
+ \diam(\Omega)^{\alpha} \ab{B^{T}}_{C^{\alpha}(\Omega;\R^{n \times n})}) \no{g}_{\Lambda_{b}^{\alpha}(\Omega;\rn)}.
\end{multline*}
As $g$ is arbitrary,
the claim follows by Corollary \ref{cor:duality-local-hardy} and a change of variables.
\end{proof}

Finally, for notational convenience, 
we define the projection to divergence free vector fields as follows.
\begin{definition}
\label{def:T-operator}
Let $0<\lambda\leq \Lambda < \infty$.
Let $P$ be a cube and let the smallest singular value of $A \in \R^{n\times n}$ be positive.
For $g \in L^{2}(S_{A}(P))$,
we let $T_{P,A}(g)$ be the unique $W^{1,2}_{0}(S_{A}(P))$ solution to 
\[
 \dive A^{T} \nabla T_{P,A}(g) = \dive g.
\]  
\end{definition}
It follows from the Lax--Milgram theorem that $T_{P,A}$ is well-defined.
All estimates in this section apply to $\nabla T_{P,A}$.

\section{Equations with H\"older continuous coefficients}
\label{sec:hoelder-coefficients} 

The core of the iteration leading to an interior H\"older estimate is the following bound for a duality pairing.
In this section,
we consider a fixed pair of positive ellipticity parameters $0<\lambda < \Lambda$.
We denote by $\mathcal{F}(Q)$ the family of the interiors of half-open cubes 
$P$ that are obtained by partitioning a minimal half open cube containing the open cube $Q$ and that satisfy $|P| = 2^{-n N}|Q|$,
where $N = N(\lambda,\Lambda)$ is large enough so that $K_{\lambda,\Lambda}P \subset (11/10) Q$.  
Also,
recall Definition \ref{def:normalized-bumps} of the bump functions appearing in the statement
and Definition \ref{def:T-operator} of the operator $T$.

\begin{lemma}
\label{lemma:iterable}
Let $0 < \lambda \le \Lambda< \infty$, $K = K_{\lambda,\Lambda}$ and $D \ge 0$. Let $\alpha \in [0,1)$
and set $p = n/(n + \alpha)$.
Let $Q_0$ be a cube; let the measurable function $A : K Q_0 \to \R^{n \times n}$ satisfy $\sigma(A(x)) \subset [\lambda,\Lambda]$ for all $x \in KQ_0$, and let $B \in C^{\alpha}(KQ_0;\R^{n \times n})$
be such that 
\[
\ell(KQ_0)^{-\alpha} \no{B}_{L^{\infty}(KQ_0;\R^{n \times n})} + \ab{B}_{C^{\alpha}(KQ_0;\R^{n \times n})} \le  D.
\]
Assume that $u \in W^{1,2}(KQ_0)$ satisfies for all test functions $\eta \in C^{\infty}_c(KQ_0)$
\[
\int A(x) \nabla u(x) \cdot \nabla \eta(x) \, dx = 0.
\] 

Then,
if $g \in L^{2}(Q_0;\rn)$,
it holds 
\begin{multline*}
\left \lvert \int_{Q_0} B \nabla u  \cdot g\, dx  \right \rvert  
\le C  D |Q_0|^{1+\alpha/n} \left( \sint_{KQ_0} |\nabla u |^{2} \, dx \right)^{1/2} \no{g}_{h_z^{p}(Q_0;\rn)} \\
+ \left \lvert \sum_{P \in \mathcal{F}(Q_0)} \int_{S_{A_P}(P)} (A-A_{P})\nabla u   \cdot 1_{S_{A_P}(P)}\nabla T_{P,A_P}(1_{Q_0}\psi_{P} B^{T} g )\, dx    \right \rvert  
\end{multline*}
where $\psi_{P} \in \mathcal{A}_{2P}$, $C = C(n,\alpha,\lambda,\Lambda)$, and 
\[
A_P = \sint_{3P} A(x) \, dx.
\]
\end{lemma}
 
\begin{proof}
Let $\psi_{Q(0,1)}$ be a smooth function with 
\[
1_{Q(0,1)} \le \psi_{Q(0,1)} \le 1_{Q(0,2)}
\] 
and for a cube $P = P(x_0,r_0)$ let 
\[
\tilde{\psi}_{P}(x) = \psi_{Q(0,1)} \left( \frac{x-x_0}{r_0} \right)
\] 
and 
\[
\psi_{P}(x) := \frac{\tilde{\psi}_P}{\sum_{P \in \mathcal{F}(Q_0)} \tilde{\psi}_{P}} 
\]
so that 
\[
\sum_{P \in \mathcal{F}(Q_0)} \psi_{P} = 1
\]
in $Q_0$ and the functions $\psi_{P}$ satisfy 
\[
0 \le \psi_{P} \le 1, \quad |\partial^{\gamma} \psi_{P}| \le C_{n,\gamma} |Q_0|^{-|\gamma|/n}
\]
for all $\gamma \in \N^{n}$.

Now
\begin{align*}
\bigg \lvert\int_{Q_0} &  B(x)  \nabla u(x) \cdot g(x) \, dx \bigg \rvert \\
&
\le \left \lvert\sum_{P{\in \mathcal{F}(Q_0)}} \int \psi_{P}(x) B(x) \nabla u_{P}(x) \cdot g(x)1_{Q_0}(x) \, dx  \right \rvert \\
&\ + \left \lvert\sum_{P{{\in \mathcal{F}(Q_0)}}} \int \psi_{P}(x) B(x) [\nabla u(x)- \nabla u_{P}(x)] \cdot g(x) 1_{Q_0}(x) \, dx \right \rvert \\ 
&= \I + \II  
\end{align*}
where we define $u_{P} \in u +  W^{1,2}_{0}(S_{A_P}(P))$ as the function solving 
\[
-\dive A_{P} \nabla u_{P} = 0, \qquad A_{P} := \sint_{3P} A(x) \, dx
\]
in the weak sense.

To estimate $\I$,
we apply the Hardy--H\"older duality from Corollary \ref{cor:duality-local-hardy} 
to estimate
\begin{multline}
\label{eq:I-start}
\left \lvert \int \psi_{P}(x) B(x) \nabla u_{P}(x) \cdot g(x)1_{Q_0}(x) \, dx \right \rvert \\
\lesssim  |P|^{1+\alpha/n} \no{ B \nabla u_{P}}_{\Lambda_r^{\alpha}(2P)} \no{\psi_{P} g1_{Q_0}}_{h_{z}^{p}(2P)} .  
\end{multline}
By the assumption on $B$ and 
by Lemma \ref{lemma:hoelder-constant-no-RHS}  
\begin{align*}
\ab{ B }_{C^{\alpha}(2P)}& \lesssim D, \\
\ab{\nabla u_{P}}_{C^{\alpha}(2P)}& \lesssim \frac{1}{\ell(P)^{\alpha}} \left(\sint_{S_{A_P}(P)} |\nabla u_{P}(x)|^{2} {\, dx} \right)^{1/2}   
\end{align*}  
and as 
\[
\no{ \nabla u_{P}}_{L^{\infty}(2P)} \le \no{  \nabla u_{P} - \la{\nabla u_{P}}_{2P}}_{L^{\infty}(2P)} + \left(\sint_{2P} |\nabla u_{P}(x)|^{2} {\, dx} \right)^{1/2}, 
\] 
we also have
\begin{align*}
\no{ B }_{L^{\infty}(2P)}& \lesssim \ell(Q_0)^{\alpha} D ,   \\
\no{  \nabla u_{P}}_{L^{\infty}(2P)}& \lesssim  \left(\sint_{S_{A_P}(P)} |\nabla u_{P}(x)|^{2} {\, dx} \right)^{1/2}   
\end{align*} 
so that all in all  
\begin{equation*}
\ab{ \ell(P)^{\alpha} B \nabla u_{P}}_{C^{\alpha}(2P)}
\lesssim \ell(Q_0)^{\alpha} D |P|^{-1/2} \no{ \nabla u_{P}}_{L^{2}(S_{A_P}(P))}.
\end{equation*}
Here, 
because $u_{P}$ solves $-\dive A_P \nabla u_P = 0$ with boundary values of $u$,
we have
\[
 |P|^{-1/2}\no{ \nabla u_{P}}_{L^{2}(S_{A_P}(P))} 
\lesssim  |P|^{-1/2} \no{ \nabla u}_{L^{2}(S_{A_P}(P))} \lesssim |Q_0|^{-1/2} \no{ \nabla u}_{L^{2}(KQ_0)}. 
\]
Hence the factor with $\nabla u_{P}$ in \eqref{eq:I-start} is bounded by
\[
C D \ell(Q_0)^{\alpha} |Q_0|^{-1/2} \no{ \nabla u}_{L^{2}(KQ_0)}.
\]

To estimate the factor with $g$ in \eqref{eq:I-start},
we use the definition of Hardy space with zero trace,
the fact $\supp \psi_P \subset 2P$,
and Lemma \ref{lemma:cut-off-hardy}.
We see that
\begin{align*}
\no{\psi_P g1_{Q_0}}_{h_{z}^{p}(2P )} 
&= \left( \frac{1}{|2P|}  \int \mathcal{M}_{\ell(P)}( 1_{2P} \psi_P g1_{Q_0} )(x)^{p} \, dx \right)^{1/p} \\
&= \left( \frac{1}{|2P|} \int \mathcal{M}_{\ell(P)}( \psi_P g1_{Q_0} )(x)^{p} \, dx \right)^{1/p} \\
&\le \left( {\frac{|Q_0|}{|2P|}}\right)^{1/p} \no{\psi_P g 1_{Q_0}}_{h_{z}^{p}(Q_0)} 
\le C \no{g}_{h_{z}^{p}(Q_0)}.
\end{align*}  
Hence 
\begin{multline*}
\I \le C D \ell(Q_0)^{\alpha}  |Q_0|^{-1/2} \no{\nabla u}_{L^{2}(KQ_0)} \no{g}_{h_{z}^{p}(Q_0)}
 \sum_{P\in \mathcal{F}(Q_0) } |P| \\
\le C D  \ell(Q_0)^{\alpha} |Q_0|  \left( \sint_{KQ_0}|\nabla u(x)|^{2} \, dx \right)^{1/2} \no{g}_{h_{z}^{p}(Q_0)}
\end{multline*}
which is the desired estimate for $\I$.

We turn the attention to $\II$. 
Denote $w_P = u - u_{P}$ so that $w_P \in W^{1,2}_{0}(S_{A_P}(P))$.
Then 
\begin{multline*}
\int  \psi_{P}(x) B(x) \nabla w_P(x)\cdot g(x)1_{Q_0}(x) \, dx  \\
= \int_{S_{A_P}(P)} \nabla w_P(x) \cdot A_{P}^{T}\nabla T_{P,A_P}(\psi_{P} B^{T}g1_{Q_0})(x)  \, dx  
\end{multline*}
by the definition of $T_{P,A_P}$.
Indeed,
\[
\dive (f - A_{P}^{T} \nabla T_{P,A_P}(f) ) = \dive f - \dive f = 0
\]
holds for all $f \in L^{2}(S_{A_P}(P);\rn)$ as an identity in $W^{-1,2}(S_{A_P}(P))$.
Further,
we know that 
\[
-\dive A_{P} \nabla w_P = - \dive (A_{P}-A) \nabla u 
\]
holds as an identity in $W^{-1,2}(S_{A_P}(P))$
and so by the weak formulation of the equation
\begin{multline*}
\int_{S_{A_P}(P)} \nabla w_P(x) \cdot A_{P}^{T}\nabla T_{P,A_P}(\psi_{P} B^{T}g1_{Q_0})(x)  \, dx  \\
= \int_{S_{A_P}(P)} (A_{P}-A(x)) \nabla u(x) \cdot \nabla T_{P,A_P}(\psi_{P} B^{T}g1_{Q_0})(x)  \, dx  
\end{multline*}
which is the second term on the right hand side of the claimed inequality.
\end{proof}

The lemma above can be iterated to smaller and smaller scales.
This gives an estimate on the duality pairing of $\nabla u$ and a test function,
from which we will be able to infer both H\"older and supremum estimates.
Note the the seminorm of $C^{\alpha}(Q;\R^{n \times n})$ is not dilation invariant,
and hence we see the quantity $\ell(Q)^{\alpha}|A|_{C^{\alpha}(Q;\R^{n \times n})}$ appear in the estimates. That product is dilation invariant.

\begin{theorem}
\label{theorem:duality-clean}
Let $0 < \lambda \le \Lambda< \infty$. Let $\alpha \in (0,1)$
and set $p = n/(n + \alpha)$.
Let $Q_0$ be a cube and let $A \in C^{\alpha}(KQ_0;\R^{n \times n})$
satisfy $\sigma(A(x)) \subset [\lambda,\Lambda]$ for all $x \in KQ_0$.
Assume that $u \in W^{1,2}(KQ_0)$ satisfies for all test functions $\eta \in C^{\infty}_c(KQ_0)$
\[
\int A(x) \nabla u(x) \cdot \nabla \eta(x) \, dx = 0.
\] 

Then,
if $g \in L^{2}(KQ_0;\rn)$
and $\psi_0 \in C^{\infty}_{c}(Q_0)$ satisfies 
\[
0 \le \psi_0 \le 1,
\]
it holds 
\[
\left \lvert \int_{Q_0} \psi_{0}(x)  \nabla u(x) \cdot g(x) \, dx \right \rvert  
\le C |Q_0| \left( \sint_{KQ_0} | \nabla u(x)|^{2} \, dx \right)^{1/2} \no{g}_{h_{z}^{p}(KQ_0;\rn)}
\]
where 
\[
C= C(n,p,\lambda,\Lambda,  \ell(Q_0)^{\alpha}\ab{A}_{C^{\alpha}(KQ_0;\R^{n \times n})} , \psi_0 ).
\] 
\end{theorem}

\begin{proof} 
For a family of cubes $\mathcal{Q}$,
we define the operation 
\[
\tilde{\mathcal{F}}(\mathcal{Q}) :=  \bigcup_{Q \in \mathcal{Q}}\mathcal{F}(Q) 
\]
with $\mathcal{F}(Q)$ defined as in the beginning of this Section \ref{sec:hoelder-coefficients}.
Starting from the initial cube $Q_0$,
we set 
\[
    \tilde{\mathcal{F}}^{0}(Q_0) := \{Q_0\} , \quad 
    \tilde{\mathcal{F}}^{k}(Q_0) := \tilde{\mathcal{F}}(\tilde{\mathcal{F}}^{k-1}(Q_0)), \quad k \ge 1  .
\]
For $P \in \tilde{\mathcal{F}}^{k}(Q_0)$ and for $j \in \{0,\ldots,k-1\}$,
we choose a parent $P^{j} \in \tilde{\mathcal{F}}^{j}(Q_0)$ such that $P \subset P^{j}$ and $P$ is obtained from $P^{j}$ in the previously described subdivision. 
We let $\psi_{P^{j}}$ be the function as in Lemma \ref{lemma:iterable}.
We set 
\begin{align*}
\mathcal{O}_{P,0} g &= \psi_{0}g \\ 
\mathcal{O}_{P, 1} g &= 1_{S_{A_{P^{1}}}(P^{1})} \nabla T_{P^{1},A_{P^{1}}}( \psi_{P^{1}} \psi_{0} g), \\
\mathcal{O}_{P,j+1} g &=  1_{S_{A_{P^{j+1}}}(P^{j+1})} \nabla T_{P^{j+1},A_{P^{j+1}}}(\psi_{P^{j+1}}(A-A_{ P^{j}})^T \mathcal{O}_{P,j} g),  
\end{align*}
for $1 \le  j \le k-1$. 
Iterating Lemma \ref{lemma:iterable},
we obtain the estimate  
\begin{multline}
\label{eq:in-iteration-hoelder}
\left \lvert \int_{Q_0} \psi_{0}(x)\nabla u(x) \cdot g(x) \, dx \right \rvert 
\le C  |A|_{C^{\alpha}(KQ_0;\R^{n \times n})} |KQ_0|^{1+\alpha/n} \\
\times
\sum_{k=0}^{\infty} 2^{-kN(\alpha + n)}  \sum_{P \in \tilde{\mathcal{F}}^{k}(Q_0)}  
\left( \sint_{KP} | \nabla u(x)|^{2} \, dx \right)^{1/2} \no{\mathcal{O}_{P,k} g}_{h_z^{p}(KP)}.
\end{multline}
Here we also used the bound $\no{\mathcal{O}_{P,k} g}_{h_z^{p}(S_{A_P}(P))} \lesssim \no{\mathcal{O}_{P,k} g}_{h_z^{p}(KP)}$ immediately following from \eqref{cube-containing-polytope}.

Using Definition \ref{def:T-operator} and
Lemma \ref{lemma:hardy-constant},  
we see that for $P \in \tilde{\mathcal{F}}^{k}(Q_0)$ and $j \in \{1, \ldots, k-1\}$
\begin{multline}
\label{eq:hoelder-iteration-step}
\no{\mathcal{O}_{P,j+1} g}_{h_{z}^{p}(KP^{j+1})}
\le C  ( \no{ (A-A_{P^{j}})^T}_{L^{\infty}(KP^{j+1};\R^{n \times n})} \\
+ \ell(KP^{j+1})^{\alpha} \ab{  (A-A_{P^{j}})^T}_{C^{\alpha}(KP^{j+1};\R^{n \times n})})\no{ \psi_{ P^{j+1}}  \mathcal{O}_{P,j} g}_{h_{z}^{p}(KP^{j+1})}.
\end{multline}
Further,
by Lemma \ref{lemma:cut-off-hardy}
and the definition of $h_z^{p}$ norm
\[
\no{ \psi_{ P^{j+1}}  \mathcal{O}_{P,j} g}_{h_{z}^{p}(KP^{j+1})} 
\le C \no{   \mathcal{O}_{P,j} g}_{h_{z}^{p}(KP^{j})}.
\]
Hence for all $j \in \{0, \ldots, k-1\}$,
the left hand side of \eqref{eq:hoelder-iteration-step} is bounded by
\[
C 2^{-N(j+1) \alpha}  \ab{A}_{C^{\alpha}(KP^{j+1})} \ell(Q_0)^{\alpha}
\no{ \mathcal{O}_{P,j} g}_{h_{z}^{p}(KP^{j})} .
\] 
Iterating this inequality, we get 
\[
\no{\mathcal{O}_{P,k} g}_{h_z^{p}(KP)} 
\le C^{k} 2^{- k^{2} \alpha } \ell(Q_0)^{k \alpha} \ab{A}_{C^{\alpha}(KQ_0)}^{k} \no{  g}_{h_{z}^{p}(KQ_0)} ,
\]
where we have used the lower bound
\[
\sum_{j=0}^{k}j = \frac{k(k+1)}{2} > \frac{k^{2}}{2}
\]
for the arithmetic sum.
Trivially also  
\[
\left( \sint_{KP} | \nabla u(x)|^{2} \, dx \right)^{1/2} 
\le 2^{Nnk/2}  \left( \sint_{KQ_0} | \nabla u(x)|^{2} \, dx \right)^{1/2} 
\]
so that the right hand side of \eqref{eq:in-iteration-hoelder} becomes bounded by
\begin{multline*}
C |Q_0|  \ell(Q_0)^{ \alpha} \ab{A}_{C^{\alpha}({KQ_0})}\left( \sint_{{KQ_0}} | \nabla u(x)|^{2} \, dx \right)^{1/2} \\
\times
\no{g}_{h_{z}^{p}(KQ_0)}  \sum_{k=0}^{\infty} [  C 2^{- k \alpha } \ell(Q_0)^{ \alpha} \ab{A}_{C^{\alpha}({KQ_0})} ]^{k}   
\end{multline*}
The sum converges,
and we see that it satisfies the claimed dependency on $\ell(Q_0)^{\alpha}\ab{A}_{C^{\alpha}({KQ_0})}$.
\end{proof}

Deducing the traditional H\"older bound
and the $L^{\infty}$ bound from the duality pairing estimate is now straigthforward,
taking advantage of the results recalled in Section \ref{sec:hardy-spaces}.

\begin{corollary}
\label{cor:holder-noncons-holder}
Let $0 < \lambda \le \Lambda< \infty$ and let $K = K_{\lambda,\Lambda}$. Let $\alpha \in (0,1)$.
Let $Q_0$ be a cube and let $A \in C^{\alpha}(KQ_0;\R^{n \times n})$
be such that $\sigma(A(x)) \subset [\lambda,\Lambda]$ for all $x \in KQ_0$.
Assume that $u \in W^{1,2}(KQ_0)$ satisfies for all test functions $\eta \in C^{\infty}_c(KQ_0)$
\[
\int A(x) \nabla u(x) \cdot \nabla \eta(x) \, dx = 0.
\] 
Then for all $x,y \in \frac{1}{2} Q_0$
\[
|\nabla u(x) - \nabla u(y)| \le C \left( \frac{|x-y|}{\ell(Q_0)} \right)^{\alpha} \left( \sint_{KQ_0} |\nabla u (z)|^{2} \, dz \right)^{1/2}
\]
where $C = C(n,p,\lambda, \Lambda, \ell(Q_0)^{\alpha} |A|_{C^{\alpha}(KQ_0; \R^{n \times n})} )$.
\end{corollary}
\begin{proof}
By Theorem \ref{theorem:duality-clean}, Corollary \ref{cor:duality-local-hardy} and Theorem 2.7 in \cite{MR1223705} (density of $L^{2}$ in $h^{p}_z$),
taking $\psi_0 \in C_c^{\infty}(Q_0)$ that is identically one in $\frac{1}{2}Q_0$,
we see that
\[
\no{\psi_{0} \partial_i u }_{\Lambda_r^{\alpha}(Q_0)}
\le C \left( \sint_{KQ_0} | \nabla u(x)|^{2} \, dx \right)^{1/2} .
\]  
This together with 
\[
|\partial_i u|_{C^{\alpha}(2^{-1} Q_0)} \le |\psi_{0} \partial_i u|_{C^{\alpha}(Q_0)} 
\le \no{\psi_{0} \partial_i u }_{\Lambda_{r}^{\alpha}(Q_0)}
\]
implies the theorem.
\end{proof}

\begin{corollary}
\label{cor:holder-noncons-linfty}
Let $0 < \lambda \le \Lambda< \infty$ and let $K = K_{\lambda,\Lambda}$. Let $\alpha \in (0,1)$.
Let $Q_0$ be a cube and let $A \in C^{\alpha}(KQ_0;\R^{n \times n})$
be such that $\sigma(A(x)) \subset [\lambda,\Lambda]$ for all $x \in KQ_0$.
Assume that $u \in W^{1,2}(KQ_0)$ satisfies for all test functions $\eta \in C^{\infty}_c(KQ_0)$
\[
\int A(x) \nabla u(x) \cdot \nabla \eta(x) \, dx = 0.
\] 
Then 
\[
\sup_{x \in \frac{1}{2} Q_0} |\nabla u(x)| \le C   \left( \sint_{KQ_0} |\nabla u (z)|^{2} \, dz \right)^{1/2}
\]
where $C = C(n,\alpha,\lambda,\Lambda, \ell(Q_0)^{\alpha} |A|_{C^{\alpha}(KQ_0; \R^{n \times n})}  )$.
\end{corollary}
\begin{proof}
By Theorem \ref{theorem:duality-clean},
we have for $p = n/(n+\alpha)$
\[
\left \lvert \int_{Q_0} \psi_{0}(x)  \nabla u(x) \cdot g(x) \, dx \right \rvert  \\
\le C |Q_0|  \left( \sint_{KQ_0} | \nabla u(x)|^{2} \, dx \right)^{1/2} \no{g}_{h_{z}^{p}(KQ_0)} .
\]
Writing $x_0 = c(Q_0)$ and $s_0 = \frac{K}{2}\ell(Q_0)$,
we have  
\begin{multline*}
\no{g}_{h_{z}^{p}(KQ_0)} = \no{g \circ \delta_{x_0,s_0}}_{h_{z}^{p}(Q(0,1))}
\le \no{ \psi g \circ \delta_{x_0,s_0}}_{h^{p}(\rn)} \\
\le  \left( \int_{Q(0,4)} \mathcal{M}_{2} (\psi g \circ \delta_{x_0,s_0} )(x) ^{p} \, dx \right)^{1/p}
\end{multline*}
when $\psi$ is a bump function localized in $Q(0,2)$.

By Kolmogorov's inequality (Lemma 5.12 in \cite{MR1800316}),
\begin{multline*}
\left( \int_{Q(0,4)} \mathcal{M}_{2} (\psi g \circ \delta_{x_0,s_0} )(x) ^{p} \, dx \right)^{1/p} \\
\le C \int |\psi(x)g( \delta_{x_0,s_0} x)| {\, dx} \le C \no{g}_{L^{1}(2KQ_0)}.
\end{multline*}
Taking supremum over all $g$ with $\no{g}_{L^{1}(\rn)} = 1$,
we see that the claim follows.

\end{proof}

\subsection{A sparse estimate}
Next we discuss a sparse estimate in the context of classical Schauder theory.
In \cite{MR4794496},
estimates as in Corollary \ref{cor:holder-noncons-linfty}
were shown to imply sparse bound estimates.
Using the duality theory from Section \ref{sec:hardy-spaces},
we can state a similar principle in the setting of H\"older spaces,
now using Corollary \ref{cor:holder-noncons-holder}. 

\begin{definition}
\label{def:partition-bumps}
Let $Q$ be a cube, $M > 0$ and let $N_0$ be as in Definition \ref{def:normalized-bumps}.
We define $\mathcal{A}_{Q,1}(M)$ as the family of those $\varphi \in C_{c}^{\infty}(Q)$
with 
\[
|\partial^{\gamma} \varphi| \le (M \ell(Q))^{-|\gamma|} 
\]
for $|\gamma| \le N_0$.
\end{definition}

\begin{lemma}
\label{lemma:iterable-sparse-schauder}
Let $0 < \lambda \le \Lambda< \infty$. 
Let $\alpha \in (0,1)$
and set $p = n/(n + \alpha)$.
Then there exist $N$ and $C$ such that the following holds.

Let $Q$ be a cube; let $A \in C^{\alpha}(6Q;\R^{n \times n})$ be such that $\sigma(A(x)) \subset [\lambda,\Lambda]$ for all $x \in 6Q$ and let $\varepsilon> 0$.
Assume that $\varphi_{Q} \in \mathcal{A}_{2Q,1}(N)$; 
$F \in C^{\infty}(6Q;\rn)$,
and assume that $u \in W^{1,2}(6Q)$ satisfies for all test functions $\eta \in C^{\infty}_c(6Q)$
\[
\int A(x) \nabla u(x) \cdot \nabla \eta(x) \, dx = \int F(x) \cdot \nabla \eta(x) {\, dx}.
\] 

Then,
given $g \in C^{\infty}(6Q)$, 
there exists a family $\mathcal{G}(Q)$ of pairwise disjoint dyadic subcubes of $2Q$
such that 
\[
\left \lvert \bigcup_{P \in \mathcal{G}(Q)} P \right \rvert \le \varepsilon |Q|
\]
and 
\begin{align*}
\left \lvert \int_{2Q} \varphi_{Q}(x) \nabla u(x) \cdot g(x) \, dx \right \rvert 
&\le  C  |Q| \left( \sint_{6Q} |\nabla u(x)|^{2} \, dx \right)^{1/2} \no{g}_{h_r^{p}(4Q)} \\
& \quad + \sum_{P \in \mathcal{G}(Q)} \left \lvert \int_ {2P}  \varphi_{P}(x) \nabla w_P(x) \cdot g(x) \, dx \right \rvert 
\end{align*}
where $\varphi_{P} \in \mathcal{A}_{2P,1}(N)$ and
where $w_P \in W^{1,2}(3P)$ satisfies for all test functions $\eta \in C^{\infty}_c(3P)$
\[
\int A(x) \nabla w_P(x) \cdot \nabla \eta(x) \, dx = \int F(x) \cdot \nabla \eta(x) \,  dx.
\] 
\end{lemma}

\begin{proof}
Define the auxiliary maximal functions 
\[
S_1(x) = \sup_{P \subset 3Q} \sint_{3P} \varphi_Q(x) |\nabla u(x)|^{2} \, dx , 
\quad S_2(x) =  \sup_{P \subset 3Q} \no{g}_{h_{r}^{p}(2P)}
\]
where the suprema are over all dyadic subcubes of $3Q$.
Clearly both $S_1$ and $S_2$ are lower semicontinuous functions.
For a constant $C_0$ large enough, to be fixed later, set 
\begin{align*}
E_1 &:= \{ x \in 3Q : S_1(x) > C_0 \la{ |\nabla u|^{2}}_{6Q}  \}, \\
E_2 &:= \{ x \in 3Q : S_2(x) > C_0 \no{g}_{h_{r}^{p}(4Q)}  \}.
\end{align*}
By the Hardy--Littlewood maximal function theorem,
\[
|E_1| \le \frac{C}{C_0 \la{  |\nabla u|^{2}}_{6Q}} \int_{6Q} |\nabla u(x)|^{2} \, dx \le \frac{C|Q|}{C_0}.
\]
Choosing $C_0$ large enough, only depending on $n$ and $\varepsilon > 0$,  
we see that 
\[
|E_1| \le \varepsilon |Q| / 2.
\]

To get a similar estimate for $|E_2|$,
consider the Whitney decomposition of $E_2$.
First, for $j \ge 0$,
let $\tilde{W}_{j,2}$ be the family of all dyadic subcubes $P$ of $3Q$
such that
\begin{center}
 $\ell(P) = 3 \cdot 2^{-j-6}\ell(Q)$ and $P \cap \{x \in E_2  : 2^{-j-1} \ell(Q) < \dist_{\infty}(x, E_2^{c}) \le 2^{-j} \ell(Q) \}\ne \varnothing$ .
\end{center}
Then we let $\mathcal{W}_2$ be the family of maximal elements in $\bigcup_{j \ge 0} \tilde{\mathcal{W}}_{j,2}$.
We let $\mathcal{W}_{j,2} = \mathcal{W}_{2} \cap \tilde{\mathcal{W}}_{j,2}$.
Note that for $C_0$ large enough, 
we may ensure that $W_{j,2} = \varnothing$ for $j < j_0$ for some $j_0$ only depending on $C_0$,
that is,
there are not large Whitney cubes.
Also, for cubes $P$ in $\mathcal{W}_{2}$,
we have $P \subset E_2$ and $2^{6} P \cap E_{2}^{c} \ne \varnothing$.
By the first one of these properties, 
\[
C_0 \no{g}_{h_{r}^{p}(4Q)}  
< \no{g}_{h_{r}^{p}(2P)}.
\]
By the second one,
we can find $\tilde{P}$, a dyadic subcube of $3Q$
such that $\ell(\tilde{P}) = 2^{6}\ell(P)$,
$\tilde{P} \cap E_2^{c} \ne \varnothing$,
and $2\tilde{P} \supset 2P$.
Then 
\[
\no{g}_{h_{r}^{p}(2P)} \le \left(\frac{\ell(\tilde{P})}{\ell(P)}\right)^{n/p} \no{g}_{h_{r}^{p}(2 \tilde{P})} \le C \no{g}_{h_{r}^{p}(4Q)}
\] 
where the last inequality followed from the fact that $\tilde{P}$ 
is not contained in $E_2$.

Now,
for any $P \in \mathcal{W}_{2}$,
let $\psi_P \in C_{c}^{\infty}(3P)$ be such that $\psi_{P} = 1$ in $2P$
and $|\partial^{\gamma} \psi_{P}| \le C_{\gamma} \ell(P)^{-|\gamma|}$
for all $\gamma \in \N^{n}$.
Taking $\tilde{g}$ with $\tilde{g}=g$ in $2Q$ (arbitrary Hardy extension), 
we see by Lemma \ref{lemma:cut-off-hardy} that for all $P \in \mathcal{W}_{2}$
with $P \cap 2Q \ne \varnothing$
\[
|2P| \no{g}_{h_{r}^{p}(2P)}^{p}
\le \int_{\rn} \mathcal{M}_{\ell(2P)}(\psi_{P} \tilde{g})(x) ^{p} \, dx  
\le C \int_{8P} \mathcal{M}_{\ell(2P),*} \tilde{g}(x)^{p} \, dx .  
\]
On the other hand,
by the construction of the Whitney decomposition,
we know that
\[
\sum_{P \in \mathcal{W}_2} 1_{8P} \le C
\]
for $C$ only depending on the dimension.
Hence we may conclude 
\begin{multline*}
|E_2|
\le \sum_{P \in \mathcal{W}_{2}} |P|
\le \frac{C}{C_0^{p} \no{g}_{h_{r}^{p}(4Q)}^{p}} 
\sum_{P} \int_{8 P} \mathcal{M}_{\ell(2\tilde{P}),*} \tilde{g}(x)^{p} \, dx  \\
\le \frac{C}{C_0^{p} \no{g}_{h_{r}^{p}(4Q)}^{p}}  \int_{\rn} \mathcal{M}_{\ell(4Q),*} \tilde{g}(x)^{p} \, dx .
\end{multline*}
Taking infimum over all $\tilde{g}$ as above
and applying Lemma \ref{lemma:grand-function},
we bound 
\[
\int_{\rn} \mathcal{M}_{\ell(4Q)} \tilde{g}(x)^{p} \, dx \le |4Q| \no{g}_{h_r^{p}(4Q)}^{p}
\]
and so 
\[
|E_2| \le \varepsilon |Q|/2
\] 
provided $C_0$ is large enough,
only depending on $p$, $n$ and $\varepsilon$.
This shows that $E_1 \cup E_2$ is an open set such that any family of cubes partitioning it satisfies the size estimate for $\mathcal{G}(Q)$ as in the claim of the statement.
Next we choose carefully the most suitable partition.
 
We abandon the Whitney decomposition of $E_2$, 
and we let $\mathcal{G}(Q)$
be the Whitney decomposition of $E_1 \cup E_2$,
formed by the very same argument but $E_1 \cup E_2$ in place of $E_2$.
If $P$ and $P'$ are cubes from a Whitney decomposition,
we know that if $2P \cap 2 P' \ne \varnothing$,
then $\ell(P) \sim \ell(P')$.
Consequently,
we can find smooth functions $\{\varphi_P: P \in \mathcal{G}\}$ such that 
\[
  0 \le \varphi_P \le 1_{2P}, \quad |\partial^{\gamma} \varphi_P| \le C_{\gamma}\ell(P)^{-|\gamma|}, \quad  1_{E_1 \cup E_2} \le  \sum_{P \in \mathcal{G}} \varphi_{P} \le 1_{4Q}. 
\]
Now,
we can estimate 
\begin{multline*}
\left \lvert \int_{2Q} \varphi_{Q}(x) \nabla u(x) \cdot g(x) \, dx \right \rvert 
\le  C_0|Q| \left( \sint_{6Q} |\nabla u(x)|^{2} \, dx \right)^{1/2} \no{g}_{h_r^{p}(4Q)} \\
+ \sum_{P \in \mathcal{G}}  \left \lvert\int \varphi_{P}(x) \nabla u_P(x) \cdot g(x) \, dx \right \rvert  \\
+\sum_{P \in \mathcal{G}}  \left \lvert\int \varphi_{P}(x) [\nabla u(x) - \nabla u_{P}(x)] \cdot g(x) \, dx \right \rvert \\
= \I + \II + \III
\end{multline*}
where $u_P \in W^{1,2}(3P)$ solves 
\begin{align*}
  -\dive A \nabla u_P &= 0 , \quad \trm{in 3P}, \\
  u_P - u &\in W^{1,2}_{0}(3P).
\end{align*} 

The term $\I$ is of the desired form and its bound follows from the bounds for $S_1$ and $S_2$ in the complement of $E_1 \cup E_2$.
The term $\III$ is also of the desired form.
For the term $\II$,
we apply Corollary \ref{cor:duality-local-hardy} to estimate 
\begin{equation}
\label{eq:sparse-schauder-proof-1}
\left \lvert\sint_{2P} \varphi_{P}(x) \nabla u_P(x) \cdot g(x) \, dx \right \rvert
\le C \ell (P)^{\alpha} \no{\varphi_P \nabla u_P }_{\Lambda_{z}^{\alpha}(2P)}\no{g}_{h_{r}^{p}(2P)}.
\end{equation}
Here (similarly as in the proof of Lemma \ref{lemma:iterable}),
we have 
\begin{multline*}
\ell (P)^{\alpha} \no{\varphi_P \nabla u_P }_{\Lambda_{z}^{\alpha}(2P)}
\le C \no{ \ell (P)^{\alpha}\varphi_P \nabla u_P }_{C^{\alpha}(2P)} \\
\le C \ab{\ell (P)^{\alpha} \varphi_P}_{C^{\alpha}(2P)} \no{\nabla u_P}_{L^{\infty}(2P)} +  C \ab{ \nabla u_P}_{C^{\alpha}(2P)} \no{\ell (P)^{\alpha} \varphi_{P}}_{L^{\infty}(2P)}.
\end{multline*}
By the construction of $\varphi_P$, by Corollary \ref{cor:holder-noncons-holder}, by Corollary \ref{cor:holder-noncons-linfty},  
the right hand side is bounded by 
\[
C \left( \sint_{3P} | \nabla u_{P} (x)|^{2} \, dx \right)^{1/2} \le C \left( \sint_{3P} | \nabla u (x)|^{2} \, dx \right)^{1/2}
\]
where the last step used the $L^{2}$ bound for the gradients.
The quantity above, in turn, has the desired upper bound by the Whitney property of $P$.
To estimate the other factor in \eqref{eq:sparse-schauder-proof-1},
we see that by the Whitney property of $P$
\[
\no{g}_{h_{r}^{p}(2P)} \le C \no{g}_{h_r^{p}(4Q)}.
\] 
\end{proof}

As a straigthforward application of Lemma \ref{lemma:iterable-sparse-schauder},
we get the sparse estimate displayed in the introduction.
\begin{proof}[Proof of Theorem \ref{intro-thm:sparse-schauder}]
This follows by iterating Lemma \ref{lemma:iterable-sparse-schauder}
(compare to \cite{MR4794496})
and using the estimate 
\[
\no{\nabla u}_{L^{2}(6P)} \le  C  \no{F - \la{F}_{6P}}_{L^{2}(6P)},
\]
valid whenever $u \in W^{1,2}_0(6P)$ is a weak solution to 
\[
-\dive A \nabla u = \dive F.
\]
Here we used that $\dive F = \dive (F - \la{F}_{6P})$.
\end{proof}

Finally,
we show that the classical Schauder estimate for equations with H\"older coefficients is hidden inside the sparse form.
Unfortunately, 
as is common with the sparse form arguments,
we do not recover the endpoint regularity. 

\begin{corollary}
\label{cor:schauder-from-sparse}
Let $0 < \lambda \le \Lambda< \infty$. Let $\alpha \in (0,1)$.
Let $Q$ be a cube and let $A \in C^{\alpha}(6Q;\R^{n \times n})$.
Let $u \in W_0^{1,2}(6Q)$ be a weak solution to 
\[
- \dive A(x) \nabla u(x) = \dive F(x)
\]
in $Q$ for $F \in C^{\infty}(6Q;\rn)$.

Let $\beta \in (0,\alpha)$.
Then for all $x,y \in Q$
\[
|\nabla u(x) - \nabla u (y)| 
\le C  |x-y|  ^{\beta}    \ab{F}_{C^{\beta}(6Q;\rn)}
\]
where $C = C(n,\lambda,\Lambda,\alpha, \ell(Q)^{\alpha} |A|_{C^{\alpha}(6Q;\R^{n \times n})},  \alpha - \beta)$.
\end{corollary}
\begin{proof} 
By the definition of the grand maximal function and
Lemma \ref{lemma:grand-function},
we have that for any $\tilde{g} \in L^{p}(\rn)$ with $\tilde{g} = g$ in $4Q$
and all $P \subset 2Q$
\[
\no{g}_{h_r^{p}(4P)}
\le C \left( \frac{1}{|4P|} \int_{16P} \mathcal{M}_{\ell(4P),*} \tilde{g}(x)^{p} \, dx   \right)^{1/p} .
\]
Choosing $q > p$ with $(\alpha - \beta)/n = 1/p - 1/q$
and using the definitions,
we see that
\begin{align*} 
\sum_{P \in \mathcal{G}}  |P|  &\left( \sint_{6P} |F  (x)  - \la{F}_{6P}|^{2} \, dx \right)^{1/2}\no{g}_{h_r^{p}(4P)} \\
&\le  C   \no{F}_{\Lambda_r^{\beta}(6Q)}  \sum_{P \in \mathcal{G}} |P|^{\frac{\beta- \alpha}{n}}  \left(  \int_{16P} \mathcal{M}_{\ell(4P),*} \tilde{g}(x)^{p} \, dx   \right)^{1/p} \\
&\le C \no{F}_{\Lambda_r^{\beta}(6Q)}  \sum_{P \in \mathcal{G}} |P|^{1/q} \inf_{x \in P} M (\mathcal{M}_{\ell(4Q),*} \tilde{g}^{p})(x)^{1/p} \\
&\le C \no{F}_{\Lambda_r^{\beta}(6Q)}  \left( \sum_{P \in \mathcal{G}} |P| \inf_{x \in P} M (\mathcal{M}_{\ell(4Q),*} \tilde{g}^{p})(x)^{q/p} \right)^{1/q}  \\
&\le C \no{F}_{\Lambda_r^{\beta}(6Q)}  \left( \sum_{P \in \mathcal{G}} |E_P| \inf_{x \in P} M (\mathcal{M}_{\ell(4Q),*} \tilde{g}^{p})(x)^{q/p} \right)^{1/q} \\
&\le C \no{F}_{\Lambda_r^{\beta}(6Q)}  \left( \int_{\rn} M (\mathcal{M}_{\ell(4Q),*} \tilde{g}^{p})(x)^{q/p} \, dx  \right)^{1/q} \\
 &\le C \no{F}_{\Lambda_r^{\beta}(6Q)}  \left( \int_{\rn} \mathcal{M}_{\ell(4Q),*} \tilde{g}(x)^{q} \, dx  \right)^{1/q} .
\end{align*}      
Minimizing over all $\tilde{g}$ as above and using Lemma \ref{lemma:grand-function},  
we bound the right hand side by 
\[
C |Q|^{1/q} \no{F}_{\Lambda_r^{\beta}(6Q)} \no{g}_{h_r^{q}(4Q)}.
\]
Because 
\[
\frac{1}{q} = \frac{1}{p} - \frac{\alpha-\beta}{n} 
= \frac{1}{p} - \frac{n(1/p-1)-\beta}{n} = 1 + \frac{\beta}{n}, 
\]
Corollary \ref{cor:duality-local-hardy} and Theorem 2.7 in \cite{MR1223705} (density of $L^{2}$ in $h^{p}_r$) imply
\[
\ab{\nabla u}_{C^{\beta}(Q)}
\le \no{\psi_{Q} \nabla u}_{\Lambda_z^{\beta}(2Q)}
\le C \no{F}_{\Lambda_r^{\beta}(6Q)}
\le C |F|_{C^{\beta}(6Q)}.
\]
\end{proof}

\begin{remark}
\label{remark:sharpness}
The argument above misses the endpoint $\alpha = \beta$,
in which case the H\"older continuity is known (usually attributed to \cite{MR192168}).
We do not know if there is a way to infer the endpoint Schauder estimates in analogy to $(1,1)$ sparse bound implying weak type $(1,1)$ estimate (Theorem E in \cite{MR3668591}).
In the case $\beta > \alpha$,
the conclusion of Corollary \ref{cor:holder-noncons-holder} is false. 
In particular, exponents $p < n/(\alpha+ n)$ cannot be used in Theorem \ref{intro-thm:sparse-schauder}.
We recall that an example of a solution to $-\dive A \nabla u = \dive F$ with $A \in C^{\alpha}(\R^{n})$, $F \in C^{\infty}(\R^{n})$ and $u \in C^{1,\alpha}(\R^{n}) \setminus \bigcup_{\beta > \alpha} C^{1,\beta}(\R^{n})$ can be constructed as follows (we suppose this is folklore by now):

Let $n = 2$ and $\alpha \in (0,1)$.
Denote $e_1  = (1,0)$ and $e_2 = (0,1)$. 
Let 
\begin{align*}
p(x) &= e_1 + |x|^{\alpha-1}x ,\\
q(x) &= e_1 + |x|^{\alpha-1} (x_2e_1 - x_1 e_2), \\
v(x) &= q(x) - p(x) .
\end{align*}
With this notation,
we set 
\[
A(x) = I + \frac{vp^{T} + pv^{T}}{|p|^{2}} - \frac{(v \cdot p)pp^{T}}{|p|^{4}}
\]
and define 
\[
u(x) = x_1 + \frac{|x|^{\alpha+1}}{\alpha+1}.
\]

In the vicinity of $0$,
we have that $|p| \sim |q| \sim 1$ and $|v| \sim |x|^{\alpha}$.
Hence $A$ is elliptic.
Moreover, it is $\alpha$-H\"older as the vectors $p$ and $q$ are.
The function $u$ is $C^{1,\alpha}(\rn)$ and no better.
Also $q = Ap$ is no better than $C^{\alpha}(\rn)$. 
Finally, it holds $\nabla u = p$, $Ap = q$ and $\dive q = 0$ 
so that $u$ solves the equation with zero on the right hand side.

If $\varphi \in C^{\infty}_{c}(B(0,5))$ is identically $1$ in $B(0,1)$,
we have that $\tilde{u} = u \varphi$ solves 
\[
- \dive A \nabla \tilde{u} =  - \dive  (u A \nabla \varphi) - A\nabla u \cdot \nabla \varphi.
\]
Because $\nabla \varphi = 0$ in $B(0,1)$,
the argument of the divergence on the right hand side is smooth and compactly supported.
For the same reason,
the second term on the right hand side is smooth and compactly supported.
Its integral is zero so it has a representation as a divergence of a smooth and compactly supported function (its image under the Bogovskij operator of $B(0,10)$, see for instance \cite{MR2240056} for the definition and properties).
Hence $\tilde{u}$ solves $-\dive A \nabla \tilde{u} = \dive F$ with smooth $F$
and both $\nabla \tilde{u}$ and $A \nabla \tilde{u}$ have $|x|^{\alpha}$ type beahviour at the origin.

Further, by duality,
this example also shows that for equations with $C^{\alpha}(\rn)$ coefficients,
the gradient of the solution $\nabla u$ need not be in any Hardy space $h^{p}$ when $p < n/(n+\alpha)$.
\end{remark}

\section{Equations with uniformly continuous coefficients}
\label{sec:bmo-case}
In this section,
we discuss
a simplified proof of a weaker version of Theorem \ref{theorem:duality-clean}
in the limiting case of the coefficient smoothness.
We assume that the coefficient matrix $A$ is uniformly continuous.
In this setting,
we cannot rely on Hardy space theory,
which serves as an excuse to expose the leading idea behind the proofs in Section \ref{sec:hoelder-coefficients} with minimal amount of technical difficulties.

In this section, let $\mathcal{F}(Q)$ be the family of the interiors of half-open cubes 
$P$ partitioning a half open cube containing the open cube $Q$ and satisfying $|P| = 2^{-3n}|Q|$.
We first prove a lemma analogous to Lemma \ref{lemma:iterable}.
Instead of duality of Hardy spaces,
we use H\"older's inequality.
\begin{lemma}
\label{lemma:iterable-bmo}
Let $0 < \lambda \le \Lambda< \infty$, $q \in (2,\infty)$. and $D \ge 0$.  
Let $Q$ be a cube; let the measurable function $A : 3Q \to \R^{n \times n}$ satisfy $\sigma(A) \subset [\lambda,\Lambda]$, and let $B \in L^{q}(3Q;\R^{n \times n})$
be such that 
\[
 |3Q|^{-1/q} \no{B}_{L^{q}(3Q;\R^{n \times n})} \le D.
\]
Assume that $u \in W^{1,2}(3Q)$ satisfies for all test functions $\eta \in C^{\infty}_c(3Q)$
\[
\int A(x) \nabla u(x) \cdot \nabla \eta(x) \, dx = 0.
\] 

Then,
if $g \in L^{2}(3Q;\rn)$,
it holds 
\begin{multline*}
\left \lvert \int_{Q} B(x)\nabla u(x) \cdot g(x) \, dx \right \rvert  \\
\le C  D |Q|^{1/q} \left( \sint_{3Q} |\nabla u(x)|^{2} \, dx \right)^{1/2} \no{g}_{L^{q'}(3Q;\rn)} \\
+ \left \lvert \sum_{P \in \mathcal{F}(Q)} \int_{3P} (A-A_{P})\nabla u (x) \cdot \nabla T_{3P,A_P}( 1_{P} B^{T} g )(x) \, dx \right \rvert  
\end{multline*}  
where $C = C(n,\lambda,\Lambda)$ and 
\[
A_P = \sint_{3P} A(x) \, dx.
\]
\end{lemma}
 
\begin{proof} 
Because $\mathcal{F}(Q)$ forms a partition of $Q$,
up to a set of measure zero,
it holds 
\begin{align*}
\bigg \lvert\int_{Q}  B(x) & \nabla u(x) \cdot g(x) \, dx \bigg \rvert \\
&
\le \left \lvert\sum_{ {P\in\mathcal{F}(Q)}} \int_{P} B(x) \nabla u_{P}(x) \cdot g(x) \, dx  \right \rvert \\
&\qquad  + \left \lvert\sum_{ {P\in\mathcal{F}(Q)}} \int_{P} B(x) [\nabla u(x)- \nabla u_{P}(x)] \cdot g(x) \, dx \right \rvert \\ 
&= \I + \II  
\end{align*}
where we define $u_{P} \in u +  W^{1,2}_{0}(3P)$ as the function solving 
\[
-\dive A_{P} \nabla u = 0, \qquad A_{P} := \sint_{3P} A(x) \, dx
\]
in the weak sense.

To estimate $\I$,
we apply H\"older's inequality (denoting $q' = q/(q-1)$) 
to estimate
\begin{equation*}
\left \lvert \int 1_{P}(x) B(x) \nabla u_{P}(x) \cdot g(x) \, dx \right \rvert
\le \no{B}_{L^{q}(P)} \no{ \nabla u_{P}}_{L^{\infty}(P)} \no{g}_{L^{q'}(P)} .  
\end{equation*}
Here by Lemma \ref{lemma:hoelder-constant-no-RHS} and the fact that $u_P$ solves an equation with boundary values of $u$ in $3P$
\[
\no{ \nabla u_{P}}_{L^{\infty}(P)} \le C |P|^{-1/2} \no{ \nabla u_{P}}_{L^{2}(2P)}
\le C |P|^{-1/2} \no{ \nabla u}_{L^{2}(3P)}.
\] 
Hence it holds
\begin{multline*}
I\leq C  D |3Q|^{-1/2} \no{\nabla u}_{L^{2}(3Q)} \sum_{P}|P|^{1/q} \no{g}_{L^{q'}(P)} \\
\leq C  D |Q|^{1/q} \left( \sint_{3Q} |\nabla u(x)|^{2} \, dx \right)^{1/2} \no{g}_{L^{q'}(3Q)}
\end{multline*}
which is the desired estimate.
 
We turn the attention to $\II$. 
Denote $w_P = u - u_{P}$ so that $w \in W^{1,2}_{0}(3P)$.
Then 
\begin{multline*}
\int_{P}  B(x) \nabla w_P(x)\cdot g(x) \, dx  \\
= \int_{3P} \nabla w_P(x) \cdot A_{P}^{T}\nabla T_{3P,A_P}( 1_P B^{T}g)(x)  \, dx  
\end{multline*}
by the definition of $T_{3P,A_P}$.
Indeed,
\[
\dive (f - A_{P}^{T} \nabla T_{3P,A_P}f ) = \dive f - \dive f = 0
\]
for all $f \in L^{2}(3P;\rn)$ as an identity in $W^{-1,2}(3P)$.
Further,
we know that 
\[
-\dive A_{P} \nabla w_P = - \dive (A_{P}-A) \nabla u 
\]
as an identity in $W^{-1,2}(3P)$
and so by the weak formulation of the equation
\begin{multline*}
\int_{3P} \nabla w_P(x) \cdot A_{P}^{T}\nabla T_{3P,A_P}(1_{P} B^{T}g)(x)  \, dx  \\
= \int_{3P} (A_{P}-A(x)) \nabla u(x) \cdot \nabla T_{3P,A_P}(1_{P} B^{T}g)(x)  \, dx  
\end{multline*}
which is the second term on the right hand side of the claimed inequality.
\end{proof}

Next,
we may iterate the lemma
and prove a result similar to Theorem \ref{theorem:duality-clean}
but the H\"older assumption replaced by mere uniform continuity and the conclusion featuring $L^{p}$ norm as opposed to a Hardy norm.

\begin{theorem}
\label{theorem:duality-clean-bmo}
Let $0 < \lambda \le \Lambda< \infty$ and $q \in (2,\infty)$.
Let $A$ be a uniformly continuous matrix valued function satisfying $\sigma(A) \subset [\lambda,\Lambda]$.
There exists $\delta = \delta(A,n,q) > 0$ such that the following holds.
Let $Q_0$ be a cube with $\ell(Q_0) < \delta$.
Assume that $u \in W^{1,2}(3Q_0)$ satisfies for all test functions $\eta \in C^{\infty}_c(3Q_0)$
\[
\int A(x) \nabla u(x) \cdot \nabla \eta(x) \, dx = 0.
\] 

Then, 
it holds 
\[
\left( \sint_{Q_0} | \nabla u(x)|^{q} \, dx \right)^{1/q} 
\le C    \left( \sint_{3Q_0} | \nabla u(x)|^{2} \, dx \right)^{1/2}  
\]
where 
\[
C= C(A,\delta,n,q,\lambda,\Lambda).
\] 
\end{theorem}

\begin{proof} 
For a family of cubes $\mathcal{Q}$,
we define $\mathcal{F}(\mathcal{Q}) := \bigcup_{Q \in \mathcal{Q}} \mathcal{F}(Q)$. 
Starting from the initial cube $3Q_0$ and iterating the operation $\mathcal{F}$,
for each $k \in \N$ we have the family $\mathcal{F}^{k}(3Q_0)$.
For $P \in \mathcal{F}^{k}(3Q_0)$ and for $j \in \{0,\ldots,k-1\}$,
we choose one $P^{j} \in \mathcal{F}^{j}(3Q_0)$ such that $P \subset P^{j}$.
We set 
\begin{align*}
\mathcal{O}_{P,0} g &= g \\ 
\mathcal{O}_{P, 1} g &= \nabla T_{3P^{1},A_{ P^{1}}}( 1_{ P^{1}}  g) \\
\mathcal{O}_{P,j+1} g &= \nabla T_{3P^{j+1},A_{ P^{j+1}}}(1_{ P^{j+1}}(A-A_{ P^{j}})^T \mathcal{O}_{P,j} g), \quad 1 \le  j \le k-1 . 
\end{align*}
Denote $p = q/(q-1)$.
Iterating Lemma \ref{lemma:iterable-bmo},
we obtain the estimate  
\begin{multline}
\label{eq:in-iteration-bmo}
\left \lvert \int_{Q_0}  \nabla u(x) \cdot g(x) \, dx \right \rvert \\
\le C  \sum_{k=0}^{\infty}   \sum_{P \in \mathcal{F}^{k}(3Q_0)}  |P|^{1/q}
\left( \sint_{3P} |  \nabla u(x)|^{2} \, dx \right)^{1/2} \no{\mathcal{O}_{P,k} g}_{L^{q'}(3P)}
\end{multline}
where $C$ is the constant induced by Lemma \ref{lemma:iterable-bmo}.
By uniform continuity of $A$,
we know that given $\varepsilon > 0$,
provided that $\delta = \delta(\varepsilon) > 0$ is small enough, 
then for all $P \in \bigcup_{k=0}^{\infty} \mathcal{F}(3Q_0)$
\[
\no{A-A_{P}}_{L^{\infty}(3Q_0)} \le \varepsilon .
\] 
Then, using Definition \ref{def:T-operator}, 
the classical $L^{p}$-bound for the operator $T_{P,A}$ (e.g.\ as a corollary of Lemma \ref{lemma:hoelder-constant} an Corollary 1.3 in \cite{MR4794496}), 
and the uniform continuity of $A$,
we see that for $P \in \mathcal{F}^{k}(3Q_0)$ 
\[
\no{\mathcal{O}_{P,k} g}_{L^{q'}(3P)}
\le (C \varepsilon)^{k} \no{ g}_{L^{q'}(3Q_0)}.
\]  
Trivially also  
\[
\left( \int_{3P} | \nabla u(x)|^{2} \, dx \right)^{1/2} 
\le  \left( \int_{3Q_0} | \nabla u(x)|^{2} \, dx \right)^{1/2} 
\]
so that by H\"older's inequality the right hand side of \eqref{eq:in-iteration-bmo} becomes bounded by
\[
C |Q_0|  \left( \sint_{3Q_0} | \nabla u(x)|^{2} \, dx \right)^{1/2} \left( \sint_{3Q_0} |g(x)|^{p} \, dx \right)^{1/p}   \sum_{k=0}^{\infty} [ C \varepsilon ]^{k}   
\]
for $C = C(n,p,\lambda,\Lambda)$.
Hence for $\varepsilon = \varepsilon(n,p,\lambda,\Lambda)$ small enough,
we see that the sum converges. 
Taking supremum over all $g \in L^{p}(3Q_0)$ with $\no{g}_{L^{p}(3Q_0)} \le |Q_0|^{1/p}$,
we see that the claim follows.
\end{proof}

\begin{remark}
Theorem \ref{theorem:duality-clean-bmo} together with \cite{MR4794496} 
implies sparse bounds and further local Calder\'on--Zygmund theory
for equations with uniformly continuous coefficients.
We also point out a small correction to \cite{MR4794496}:
In that paper, the smooth domains $O_P \cap \Omega$ relative 
to cubes $P$ cannot be defined as an intersection as written, 
but at small enough scales it is easy to see there exist domains as smooth as $\Omega$ contained in $3P \cap \Omega$ and $C^{1}$-Dini norms independent of $P$,
which can be used instead. 
We thank Ya (Grace) Gao from Brown University for bringing this point to our attention.
\end{remark}


\bibliography{references}

\begin{thebibliography}{10}

\bibitem{MR3887613}
L.~Ambrosio, A.~Carlotto, and A.~Massaccesi.
\newblock {\em Lectures on elliptic partial differential equations}, volume~18 of {\em Appunti. Scuola Normale Superiore di Pisa (Nuova Serie) [Lecture Notes. Scuola Normale Superiore di Pisa (New Series)]}.
\newblock Edizioni della Normale, Pisa, 2018.
\newblock Available at https://cvgmt.sns.it/media/doc/paper/1280/PDEAAA.pdf.

\bibitem{MR2384540}
P.~Auscher, F.~Bernicot, and J.~Zhao.
\newblock Maximal regularity and {H}ardy spaces.
\newblock {\em Collect. Math.}, 59(1):103--127, 2008.

\bibitem{MR4628043}
P.~Auscher and M.~Egert.
\newblock {\em Boundary value problems and {H}ardy spaces for elliptic systems with block structure}, volume 346 of {\em Progress in Mathematics}.
\newblock Birkh\"auser/Springer, Cham, 2023.

\bibitem{MR1600066}
P.~Auscher, A.~McIntosh, and P.~Tchamitchian.
\newblock Heat kernels of second order complex elliptic operators and applications.
\newblock {\em J. Funct. Anal.}, 152(1):22--73, 1998.

\bibitem{MR3897969}
C.~Benea and F.~Bernicot.
\newblock Conservation de certaines propri\'et\'es \`a{} travers un contr\^ole \'epars d'un op\'erateur et applications au projecteur de {L}eray-{H}opf.
\newblock {\em Ann. Inst. Fourier (Grenoble)}, 68(6):2329--2379, 2018.

\bibitem{MR3531367}
F.~Bernicot, D.~Frey, and S.~Petermichl.
\newblock Sharp weighted norm estimates beyond {C}alder\'{o}n-{Z}ygmund theory.
\newblock {\em Anal. PDE}, 9(5):1079--1113, 2016.

\bibitem{MR4359452}
D.~Breit, A.~Cianchi, L.~Diening, and S.~Schwarzacher.
\newblock Global {S}chauder estimates for the {$p$}-{L}aplace system.
\newblock {\em Arch. Ration. Mech. Anal.}, 243(1):201--255, 2022.

\bibitem{MR2069724}
S.-S. Byun and L.~Wang.
\newblock Elliptic equations with {BMO} coefficients in {R}eifenberg domains.
\newblock {\em Comm. Pure Appl. Math.}, 57(10):1283--1310, 2004.

\bibitem{MR1486629}
L.~A. Caffarelli and I.~Peral.
\newblock On {$W^{1,p}$} estimates for elliptic equations in divergence form.
\newblock {\em Comm. Pure Appl. Math.}, 51(1):1--21, 1998.

\bibitem{MR192168}
S.~Campanato.
\newblock Equazioni ellittiche del {${\rm II}\deg $} ordine espazi {$\mathcal{L}^{(2,\lambda )}$}.
\newblock {\em Ann. Mat. Pura Appl. (4)}, 69:321--381, 1965.

\bibitem{MR1253178}
D.-C. Chang.
\newblock The dual of {H}ardy spaces on a bounded domain in {${\bf R}^n$}.
\newblock {\em Forum Math.}, 6(1):65--81, 1994.

\bibitem{MR2515406}
D.-C. Chang, G.~Dafni, and H.~Yue.
\newblock A div-curl decomposition for the local {H}ardy space.
\newblock {\em Proc. Amer. Math. Soc.}, 137(10):3369--3377, 2009.

\bibitem{MR1223705}
D.-C. Chang, S.~G. Krantz, and E.~M. Stein.
\newblock {$H^p$} theory on a smooth domain in {${\bf R}^N$} and elliptic boundary value problems.
\newblock {\em J. Funct. Anal.}, 114(2):286--347, 1993.

\bibitem{MR3668591}
J.~M. Conde-Alonso, A.~Culiuc, F.~Di~Plinio, and Y.~Ou.
\newblock A sparse domination principle for rough singular integrals.
\newblock {\em Anal. PDE}, 10(5):1255--1284, 2017.

\bibitem{MR4475438}
G.~Dafni, C.~H. Lau, T.~Picon, and C.~Vasconcelos.
\newblock Inhomogeneous cancellation conditions and {C}alder\'on-{Z}ygmund type operators on {$h^p$}.
\newblock {\em Nonlinear Anal.}, 225:Paper No. 113110, 22, 2022.

\bibitem{MR4758469}
G.~Dafni, C.~H. Lau, T.~Picon, and C.~Vasconcelos.
\newblock Necessary cancellation conditions for the boundedness of operators on local {H}ardy spaces.
\newblock {\em Potential Anal.}, 61(1):1--11, 2024.

\bibitem{MR4665778}
C.~De~Filippis and G.~Mingione.
\newblock Nonuniformly elliptic {S}chauder theory.
\newblock {\em Invent. Math.}, 234(3):1109--1196, 2023.

\bibitem{MR1405255}
G.~Di~Fazio.
\newblock {$L^p$} estimates for divergence form elliptic equations with discontinuous coefficients.
\newblock {\em Boll. Un. Mat. Ital. A (7)}, 10(2):409--420, 1996.

\bibitem{MR3747493}
H.~Dong, L.~Escauriaza, and S.~Kim.
\newblock On {$C^1$}, {$C^2$}, and weak type-{$(1,1)$} estimates for linear elliptic operators: part {II}.
\newblock {\em Math. Ann.}, 370(1-2):447--489, 2018.

\bibitem{MR1800316}
J.~Duoandikoetxea.
\newblock {\em Fourier analysis}, volume~29 of {\em Graduate Studies in Mathematics}.
\newblock American Mathematical Society, Providence, RI, 2001.
\newblock Translated and revised from the 1995 Spanish original by David Cruz-Uribe.

\bibitem{MR2597943}
L.~C. Evans.
\newblock {\em Partial differential equations}, volume~19 of {\em Graduate Studies in Mathematics}.
\newblock American Mathematical Society, Providence, RI, second edition, 2010.

\bibitem{MR447953}
C.~Fefferman and E.~M. Stein.
\newblock {$H\sp{p}$} spaces of several variables.
\newblock {\em Acta Math.}, 129(3-4):137--193, 1972.

\bibitem{MR2240056}
M.~Gei\ss{}ert, H.~Heck, and M.~Hieber.
\newblock On the equation {${\rm div}\,u=g$} and {B}ogovski\u i's operator in {S}obolev spaces of negative order.
\newblock In {\em Partial differential equations and functional analysis}, volume 168 of {\em Oper. Theory Adv. Appl.}, pages 113--121. Birkh\"auser, Basel, 2006.

\bibitem{MR749677}
M.~Giaquinta and E.~Giusti.
\newblock Global {$C^{1,\alpha }$}-regularity for second order quasilinear elliptic equations in divergence form.
\newblock {\em J. Reine Angew. Math.}, 351:55--65, 1984.

\bibitem{MR1814364}
D.~Gilbarg and N.~S. Trudinger.
\newblock {\em Elliptic partial differential equations of second order}.
\newblock Classics in Mathematics. Springer-Verlag, Berlin, 2001.
\newblock Reprint of the 1998 edition.

\bibitem{MR523600}
D.~Goldberg.
\newblock A local version of real {H}ardy spaces.
\newblock {\em Duke Math. J.}, 46(1):27--42, 1979.

\bibitem{MR3342492}
G.~Hoepfner, J.~Hounie, and T.~Picon.
\newblock Div-curl type estimates for elliptic systems of complex vector fields.
\newblock {\em J. Math. Anal. Appl.}, 429(2):774--799, 2015.

\bibitem{MR1720770}
J.~Kinnunen and S.~Zhou.
\newblock A local estimate for nonlinear equations with discontinuous coefficients.
\newblock {\em Comm. Partial Differential Equations}, 24(11-12):2043--2068, 1999.

\bibitem{MR2900466}
T.~Kuusi and G.~Mingione.
\newblock Universal potential estimates.
\newblock {\em J. Funct. Anal.}, 262(10):4205--4269, 2012.

\bibitem{MR3625108}
M.~T. Lacey.
\newblock An elementary proof of the {$A_2$} bound.
\newblock {\em Israel J. Math.}, 217(1):181--195, 2017.

\bibitem{MR4058547}
A.~K. Lerner and S.~Ombrosi.
\newblock Some remarks on the pointwise sparse domination.
\newblock {\em J. Geom. Anal.}, 30(1):1011--1027, 2020.

\bibitem{Meyers1963}
N.~{Meyers}.
\newblock An \({L}^ p\)-estimate for the gradient of solutions of second order elliptic divergence equations.
\newblock {\em Ann. Sc. Norm. Super. Pisa Cl. Sci. (5)}, 17:189--206, 1963.

\bibitem{Meyers1975}
N.~G. {Meyers} and A.~{Elcrat}.
\newblock Some results on regularity for solutions of non-linear elliptic systems and quasi-regular functions.
\newblock {\em Duke Math. J.}, 42:121--136, 1975.

\bibitem{MR4794496}
O.~Saari, H.-Y. Wang, and Y.~Wei.
\newblock Sparse gradient bounds for divergence form elliptic equations.
\newblock {\em J. Differential Equations}, 413:606--631, 2024.

\bibitem{MR2216900}
Y.~Sun and W.~Su.
\newblock Interior {$h^1$}-estimates for second order elliptic equations with vanishing {LMO} coefficients.
\newblock {\em J. Funct. Anal.}, 234(2):235--260, 2006.

\bibitem{MR2395177}
L.~Yan.
\newblock Classes of {H}ardy spaces associated with operators, duality theorem and applications.
\newblock {\em Trans. Amer. Math. Soc.}, 360(8):4383--4408, 2008.

\end{thebibliography}
\bibliographystyle{abbrv}

\end{document}